\documentclass[12pt]{article}
\usepackage{amsmath}
\usepackage{amsfonts}
\usepackage{amssymb}
\usepackage{amsthm}
\usepackage{epsf}
\usepackage{graphicx}
\usepackage{color}

\newcommand{\R}{\mathbb{R}}
\newcommand{\Z}{\mathbb{Z}}
\renewcommand{\H}{\mathbb{H}}

\renewcommand{\S}{\mathbb{S}}
\newcommand{\T}{\mathbb{T}}
\newcommand{\M}{\mathbb{M}}
\renewcommand{\l}{\lambda}
\renewcommand{\t}{\theta}
\renewcommand{\a}{\alpha}

\newcommand{\g}{\gamma}
\newcommand{\G}{\Gamma}

\newcommand{\0}{\mathbf{0}}
\newcommand{\Ome}{\Omega}
\newcommand{\boP}{\mathcal{P}}
\newcommand{\eps}{\varepsilon}

\newcommand{\boQ}{\mathcal{Q}}

\DeclareMathOperator{\Div}{div}
\DeclareMathOperator{\Isom}{Isom}

\definecolor{darkgreen}{rgb}{0,0.5,0}

\newtheorem{theorem}{Theorem}[section]
\newtheorem{proposition}[theorem]{Proposition}

\newtheorem{remark}[theorem]{Remark}
\newtheorem{claim}[theorem]{Claim}

\title{Periodic constant mean curvature surfaces in $\H^2\times\R$}
\author{Laurent Mazet, M. Magdalena Rodr\'\i guez\thanks{Research
    partially supported by a Spanish MEC-FEDER Grant
    no. MTM2007-61775, a Regional J. Andaluc\'\i a Grant
    no. P09-FQM-5088 and a Spanish MICINN Grant no. PYR-2010-21 of the
    CEI BioTIC GENIL Proyect (CEB09-0010).} and Harold Rosenberg}

\begin{document}

\maketitle


\section{Introduction}\label{sec:intro}

A properly embedded surface $\Sigma$ in $\H^2\times\R$, invariant by a
non-trivial discrete group of isometries of $\H^2\times\R$, will be
called a periodic surface.  We will discuss periodic minimal and
constant mean curvature surfaces.  At this time, there is little
theory of these surfaces in $\H^2\times\R$ and other homogeneous
3-manifolds, with the exception of the space forms.

The theory of doubly periodic minimal surfaces (invariant by a $\Z^2$
group of isometries) in $\R^3$ is well developed. Such a surface in
$\R^3$, not a plane, is given by a properly embedded minimal surface
in $\T\times\R$, $\T$ some flat $2$-torus. One main theorem is that a
finite topology complete embedded minimal surface in $\T\times\R$ has
finite total curvature and one knows the geometry of the ends
\cite{MeRo2}. It is very interesting to understand this for such
minimal surfaces in $\M^2\times\R$, $\M^2$ a closed hyperbolic
surface.

In this paper we will consider periodic surfaces in $\H^2\times\R$.
The discrete groups of isometries of $\H^2\times\R$ we consider are
generated by horizontal translations $\phi_l$ along geodesics of
$\H^2$ and/or a vertical translation $T(h)$ by some $h>0$. We denote
by $\M$ the quotient of $\H^2\times\R$ by $G$.

In the case $G$ is the $\Z^2$ subgroup of the
isometry group generated by $\phi_l$ and $T(h)$, $\M$ is diffeomorphic
but not isometric to $\T\times \R$.  Moreover $\M$ is foliated by the
family of tori $\T(s)=(d(s)\times\R)/G$ (here $d(s)$ is an equidistant
to $\gamma$). All the $\T(s)$ are intrinsically flat and have constant
mean curvature; $\T(0)$ is totally geodesic.  In
Section~\ref{sec:alexandrov}, we will prove an Alexandrov-type theorem
for doubly periodic $H$-surfaces, i.e., an analysis of compact
embedded constant mean curvature surfaces in such a
$\M$ (Theorem~\ref{thm:alexandrov}).

The remainder of the paper is devoted to construct examples of
periodic minimal surfaces in $\H^2\times\R$.

The first example we want to illustrate is the singly periodic Scherk
minimal surface.  In $\R^3$, it can be understood as the
desingularization of two orthogonal planes. H.~Karcher \cite{Kar} has
generalized this to desingularize $k$ planes of $\R^3$ meeting along a
line at equal angles, these are called Saddle Towers. In
$\H^2\times\R$, two situations are similar to these ones: the
intersection of a vertical plane with the horizontal slice
$\H^2\times\{0\}$ and the intersection of $k$ vertical planes meeting
along a vertical geodesic at equal angles. These surfaces, constructed
in Section~\ref{sec:singly}, are singly periodic and called,
respectively, ``horizontal singly periodic Scherk minimal surfaces''
and ``vertical Saddle Towers''. For vertical intersections, the
situation is in fact more general and was treated by F.~Morabito and
the second author in \cite{MoRo}; here we give another approach which
is more direct (see also J.~Pyo~\cite{Pyo2}).

In Section~\ref{sec:2per}, we construct doubly periodic minimal
examples. The first examples we obtain, called ``doubly periodic
Scherk minimal surfaces'' bounded by four horizontal geodesics; two at
height zero, and two at height $h>\pi$. The latter two geodesics are
the vertical translation of the two at height zero. Each one of these
Scherk surfaces has two ``left-side'' ends asymptotic to two vertical
planar strips, and two ``right-side'' ends, asymptotic to the
horizontal slices at heights zero and $h$. By recursive rotations by
$\pi$ about the horizontal geodesics, we obtain a doubly periodic
minimal surface.

The other doubly periodic minimal surfaces of $\H^2\times\R$
constructed in Section~\ref{sec:2per} are analogous to some Karcher's
Toroidal Halfplane Layers of $\R^3$ (more precisely, the ones denoted
by $M_{\t,0,\pi/2},\ M_{\t,\pi/2,0}$ and $M_{\t,0,0}$ in~\cite{Rod}).
The examples we construct, also called Toroidal Halfplane Layers, are
all bounded by two horizontal geodesics at height zero, and its
translated copies at height $h>0$. Each ot these Toroidal Halfplane
Layers has two ``left-side'' ends and two ``right-side'' ends, all of
them asymptotic to either vertical planar strips or horizontal strips,
bounded by the horizontal geodesics in its boundary. By recursive
rotations by $\pi$ about the horizontal geodesics, we obtain a doubly
periodic minimal surface.  In the quotient of $\H^2\times\R$ by a
horizontal hyperbolic translation and a vertical translation leaving
invariant the surface, we get a finitely punctured minimal torus and
Klein bottle in $\T\times\R$, $\T$ some flat 2-torus.

Finally, in Section~\ref{sec:isomH2}, we construct a periodic minimal
surface in $\H^2\times\R$ analogous to the most symmetric Karcher's
Toroidal Halfplane Layer in $\R^3$ (denoted by $M_{\t,0,0}$
in~\cite{Rod}). A fundamental domain of this latter surface can be
viewed as two vertical strips with a handle attached. This piece is a
bigraph over a domain $\Omega$ in the parallelogram of the
$\R^2\times\{0\}$ plane whose vertices are the horizontal projection
of the four vertical lines in the boundary of the domain, and the
upper graph has boundary values $0$ and $+\infty$: The trace of the
surface on $\R^2\times\{0\}$ are the two concave curves in the
boundary of $\Omega$. They are geodesic lines of curvature on the
surface and their concavity makes the construction of these surfaces
delicate. We refer to \cite{Kar,MeRo2,Rod}, where they are constructed
by several methods.  The complete surface is obtained by rotating by
$\pi$ about the vertical lines in the boundary. Considering the
quotient of $\R^3$ by certain horizontal translations leaving
invariant the surface, yields finitely punctured minimal tori and
Klein bottles in $\T\times\R$.

The surface we construct in $\H^2\times\R$ will have a fundamental
domain $\Sigma$ which may be viewed as $k$ vertical strips ($k\ge 3$)
to which one attaches a sphere with $k$ disks removed. $\Sigma$ is a
vertical bigraph over a domain $\Omega\subset
\H^2\times\{0\}\equiv\H^2$; $\partial\Ome$ has $2k$ smooth arcs
$A_1,B_1,\cdots,A_k,B_k$ in that order. Each $A_i$ is a geodesic and
each $B_j$ is concave towards $\Ome$. The $A_i$'s are of equal length
and the $B_j$'s as well. The convex hull of the vertices of $\Ome$ is
a polygonal domain $\widetilde\Ome$ that tiles $\H^2$; the interior
angles of the vertices of $\widetilde{\Ome}$ are $\pi/2$. Thus
$\Sigma$ extends to a periodic minimal surface in $\H^2\times\R$ by
symmetries: rotation by $\pi$ about the vertical geodesic lines over
the vertices of $\partial\Ome$.

The surface $\Sigma_+=\Sigma \cap (\H^2\times\R^+)$ is a graph over
$\Ome$ with boundary values as indicated in
Figure~\ref{fig:domainomega} (here $k=4$).  $\Sigma_+$ is orthogonal
to $\H^2\times\{0\}$ along the concave arcs $B_j$ so $\Sigma$ is the
extension of $\Sigma_+$ by symmetry through $\H^2\times\{0\}$.

\begin{figure}[h]
\begin{center}
\resizebox{0.6\linewidth}{!}{\input{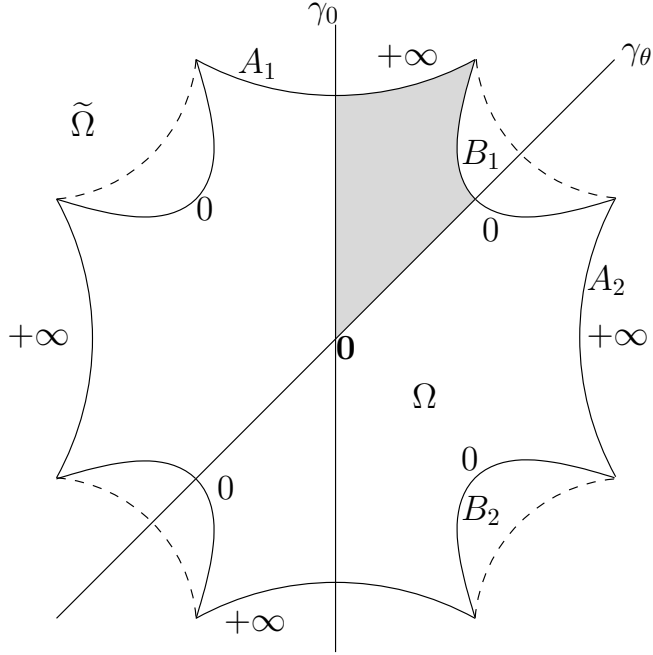}}
\caption{The domain $\Omega$ \label{fig:domainomega}}
\end{center}
\end{figure}

$\Sigma$ will be constructed by solving a Plateau problem for a
certain contour and taking the conjugate surface of this Plateau
solution. The result will be the part of $\Sigma_+$ which is a graph
over the shaded region on $\Ome$ in the
Figure~\ref{fig:domainomega}. This graph meets the vertical plane over
$\g_0$ and $\g_\t$ orthogonally, so extends by symmetry in these
vertical planes. $\Sigma_+$ is then obtained by going around $\0$ by
$k$ symmetries.


\section{Preliminaries}\label{sec:preliminaires}


\subsection{Notation}
In this paper, the Poincar\'e disk model is used for the hyperbolic
plane, i.e.
$$
\H^2=\{(x,y)\in\R^2\ |\ x^2+y^2<1\}
$$
with the hyperbolic metric $g_{-1}=\frac{4}{(1-x^2-y^2)^2} g_0$, where
$g_0$ is the Euclidean metric in $\R^2$. Thus $x$ and $y$ will be used
as coordinates in the hyperbolic space. We denote by ${\bf 0}$ the
origin $(0,0)$ of $\H^2$. In this model, the asymptotic boundary
$\partial_\infty\H^2$ of $\H^2$ is identified with the unit circle. So
any point in the closed unit disk is viewed as either a point in
$\H^2$ or a point in $\partial_\infty\H^2$.

Let $\t\in\R$. In $\H^2$, we denote by $\gamma_\t$ the geodesic line
$\{-x\cos\t+y\sin\t=0\}$ and by $\g_\t^+$ the half geodesic line from
$\0$ to $(\sin\t,\cos\t)$. We also denote by $T_\t$ the hyperbolic
angular sector $\{(r\sin u,r\cos u)\in\H^2, r\in[0,1),u\in[0,\t]\}$.

For $\mu\in(-1,1)$ we denote by $g(\mu)$ the complete geodesic of
$\H^2$ orthogonal to $\g_0$ at $q_\mu=(0,\mu)$. We have
$g(0)=\g_{\pi/2}$. We also denote $g^+(\mu)=g(\mu)\cap\{x>0\}$.

\begin{figure}
   \begin{center}
       \includegraphics[width=0.5\textwidth]{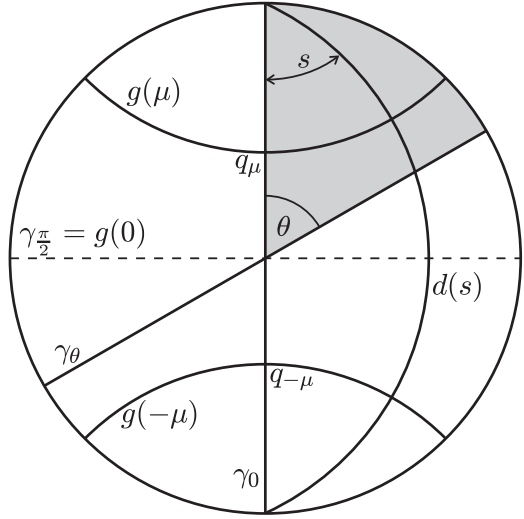}
   \end{center}
   \caption{The hyperbolic angular sector $T_\theta$ corresponds to
     the shadowed domain.}
   \label{fig:notacion}
\end{figure}

For $\t\in \R$, there exists a Killing vector field $Y_\t$ which has
length $1$ along $\g_\t$ and generated by the hyperbolic translation
along $\g_\t$ with $(\sin\t,\cos\t)$ as attractive fixed point at
infinity. For $l\in(-1,1)$, we denote by $\phi_l$ the hyperbolic
translation along $\g_\t$ with $\phi_l(\0)=(l\sin\t,l\cos\t)$.
$(\phi_l)_{l\in(-1,1)}$ is called the ``flow'' of $Y_\t$, even though
the family $(\phi_l)_{l\in(-1,1)}$ is not parameterized at the right
speed. We notice that, if $(\phi_l)_{l\in(-1,1)}$ is the flow of
$Y_0$, $g(\mu)=\phi_\mu(g(0))$.

For $\t\in\R$, there is another interesting vector field that we
denote by $Z_\t$. This vector field is the unit vector field normal to
the foliation of $\H^2$ by the equidistant lines to $\g_{\t+\pi/2}$
such that $Z_\t(\0)=(1/2) (\sin\t\partial_x+\cos\t\partial_y)$. We
notice that $Z_\t$ is not a Killing vector field. This time, we define
$(\psi_s)_{s\in \R}$ the flow of $Z_\t$ (with the right speed). If
$(\psi_s)_{s\in \R}$ is the flow of $Z_{\pi/2}$, we define
$d(s)=\psi_s(\g_0)$ for $s$ in $\R$. $d(s)$ is one of the equidistant
lines to $\g_0$ at distance $|s|$.  We remark that $Z_{\pi/2}$ is
tangent to the geodesic lines $g(\mu)$.

In the sequel, we denote by $t$ the height coordinate in
$\H^2\times\R$. Besides, we will often identify the hyperbolic plane
$\H^2$ with the horizontal slice $\{t=0\}$ of $\H^2\times\R$. The
Killing vector field $Y_\t$ and its flow naturally extend to a
horizontal Killing vector field and its flow in $\H^2\times\R$. The
same occurs for $Z_\t$ and its flow.

Besides we denote by $\pi:\H^2\times\R\rightarrow \H^2$ the vertical
projection and by $T(h)$ the vertical translation by $h$.  Given two
points $p$ and $q$ of $\H^2$ or $\H^2\times\R$, we denote by
$\overline{pq}$ the geodesic arc between these two points.


\subsection{Conjugate minimal surface}
\label{subsec:conj}

B.~Daniel \cite{Dan1} and L.~Hauswirth, R.~Sa~Earp and E.~Toubiana
\cite{HaSaTo} have proved that minimal disks in $\H^2\times\R$ have an
associated family of locally isometric minimal surfaces. In this
subsection we briefly recall how they are defined.

Let $X=(\varphi,h):\Sigma\rightarrow \H^2\times\R$ be a conformal
minimal immersion, with $\Sigma$ a simply connected Riemann
surface. Then $h$ is a real harmonic function and $\varphi=\pi\circ X$
is a harmonic map to $\H^2$. Let $h^*$ be the real harmonic conjugate
function of $h$ and $\boQ_\varphi$ be the Hopf differential of
$\varphi$. Since $X$ is conformal, we have
$$
\boQ_\varphi=-4\left(\frac{\partial h}{\partial z}\right)^2 dz^2,
$$
where $z$ is a conformal parameter on $\Sigma$. In~\cite{Dan1}
and~\cite{HaSaTo} it has been proved that, for any $\t\in\R$, there
exists a minimal immersion $X_\t=(\varphi_\t,h_\t):\Sigma\rightarrow
\H^2\times\R$ whose induced metric on $\Sigma$ coincides with the one
induced by $X$, and such that $h_\t=\cos\t h+\sin\t h^*$ and the Hopf
differential of $\varphi_\t$ is
$\boQ_{\varphi_\t}=e^{-2i\t}\boQ_\varphi$. If $N$ (resp. $N_\t$)
denotes the unit normal to $X$ (resp. $X_\t$), then $\langle
N,\partial_t\rangle=\langle N_\t,\partial_t\rangle$ (i.e. their angle
maps coincide).

All these immersions $X_\t$ are well-defined up to an isometry of
$\H^2\times\R$. The immersion $X_{\pi/2}$ is called the conjugate
immersion of $X$ (and $X_{\pi/2}(\Sigma)$ is usually called conjugate
minimal surface of $X(\Sigma)$), and it is denoted by $X^*$.

The data for the conjugate surface are the same as for $X(\Sigma)$,
except that one rotates $S$ and $T$ by $\pi/2$: $S^* = J S$, and $T^*=
J T$. Here $S$ (resp. $S^*$) denotes the symmetric operator on
$\Sigma$ induced by the shape operator of $X(\Sigma)$ (resp.
$X^*(\Sigma)$); $T$ (resp. $T^*$) is the vector field on $\Sigma$ such
that $dX(T)$ (resp. $dX^*(T^*)$) is the projection of $\partial_t$ on
the tangent plane of $X(\Sigma)$ (resp.  $X^*(\Sigma)$); and $J$ is
the rotation of angle $\pi/2$ on $T\Sigma$. See~\cite{Dan1} for more
details.

For $C$ a curve on $\Sigma$, the normal curvature of $C$ in the
surface $X(\Sigma)$ is $-\langle C', S(C')\rangle$, and the normal
torsion is $\langle J(C'), S(C')\rangle$. Thus the normal torsion of
$C$ on the conjugate surface $X^*(\Sigma)$ is minus the normal
curvature of $C$ on $X(\Sigma)$, and the normal curvature of $C$ on
$X^*(\Sigma)$ is the normal torsion of $C$ on $X(\Sigma)$. In
particular, if $C$ is a vertical ambient geodesic on $X(\Sigma)$, then
$C$ is a horizontal line of curvature on the conjugate surface
$X^*(\Sigma)$ whose geodesic curvature in the horizontal plane is the
normal torsion on $X(\Sigma)$. Arguing similarly, we get that the
correspondence $X \leftrightarrow X^*$ maps:
\begin{itemize}
\item vertical geodesic lines to horizontal geodesic curvature lines
  along which the normal vector field of the surface is horizontal;
  and
\item horizontal geodesics to geodesic curvature lines contained in
  vertical geodesic planes $\Pi$ (i.e. $\pi(\Pi)$ is a geodesic of
  $\H^2$) along which the normal vector field is tangent to $\Pi$.
\end{itemize}
Moreover, this correspondence exchanges the corresponding Schwarz
symmetries of the surfaces $X$ and $X^*$. For more definitions and
properties, we refer to \cite{Dan1,HaSaTo}.


\subsection{Some results about graphs}

In $\H^2\times\R$, there exist different notions of graphs, depending
on the vector field considered.

If $u$ is a function on a domain $\Ome$ of $\H^2$, the graph of $u$,
defined as
\[
\Sigma_u=\{(p,u(p))\ |\ p\in\Omega\} ,
\]
is a surface in $\H^2\times\R$. This surface is minimal (a vertical
minimal graph) if $u$ satisfies the vertical minimal graph equation
\begin{equation}\label{mse}
  \Div\left(\frac{\nabla u}{\sqrt{1+\|\nabla u\|^2}}\right)=0 ,
\end{equation}
where all terms are calculated with respect to the hyperbolic metric.

If $u$ is a solution of equation~\eqref{mse} on a convex domain of
$\H^2$, L.~Hauswirth, R.~Sa~Earp and E.~Toubiana have proved
in~\cite{HaSaTo} that the conjugate minimal surface $\Sigma_u^*$ of
$\Sigma_u$ is also a vertical graph.

Assume $\Omega$ is simply connected. The differential on $\Ome$ of the
height coordinate of $\Sigma_u^*$ is the closed $1$-form
\begin{equation}\label{def:omega}
  \omega_u^*(X)=\langle\frac{{\nabla u}^\perp}{\sqrt{1+\|\nabla
      u\|^2}},X\rangle_{\H^2} ,
\end{equation}
where ${\nabla u}^\perp$ is the vector $\nabla u$ rotated by
$\pi/2$. The height coordinate of $\Sigma_u^*$ is a primitive $h_u^*$
of $\omega_u^*$ and is the conjugate function of $h_u$ on
$\Sigma_u$. The formula \eqref{def:omega} comes from the following
computation. Let $h$ be the height function along the graph surface
and $h^*$ its conjugate harmonic function. Let $(e_1,e_2)$ an
orthonormal basis of the tangent space to $\H^2$ and $X=x_1e_1+x_2e_2$
a tangent vector. Then
$$
\omega_u^*(X)=d h^*(X+ \langle\nabla
u,X\rangle_{\H^2} \partial_t)=dh(N_u\wedge(X+\langle\nabla
u,X\rangle_{\H^2} \partial_t))
$$
where $N_u=(\nabla u-\partial t)/W$ (with $W=\sqrt{1+\|\nabla
u\|^2})$. If $\nabla u=u_1e_1+u_2e_2$ we have
$$
N_u\wedge(X+\langle\nabla
u,X\rangle_{\H^2} \partial_t)=\frac{u_2\langle\nabla
  u,X\rangle_{\H^2}+x_2}{W}e_1 -\frac{u_1\langle\nabla
  u,X\rangle_{\H^2}+x_1}{W}e_2+\frac{u_1x_2-u_2x_1}{W}\partial_t
$$
Thus
$$
\omega_u^*(X)=\frac{u_1x_2-u_2x_1}{W}=\langle\frac{{\nabla
u}^\perp}{\sqrt{1+\|\nabla u\|^2}},X\rangle_{\H^2}
$$

Let us now fix $\t\in\R$. Recall that $(\phi_l)_{l\in(-1,1)}$ is the
flow of the Killing vector field $Y_\t$. Let $D$ be a domain in the
vertical geodesic plane $\gamma_{\t+\pi/2}\times\R$ (this plane is
orthogonal to $\gamma_\theta$, viewed as a geodesic of $\{t=0\}$). Let
$v$ be a function on $D$ with values in $(-1,1)$. Then, the surface
$\{\phi_{v(p)}(p)\ |\ p\in D\}$ is called a $Y_\t$-graph. It is a
graph with respect to the Killing vector field $Y_\t$ in the sense
that it meets each orbit of $Y_\t$ in at most one point. If such a
surface is minimal, it is called a minimal $Y_\t$-graph.  Let $v'$ be
a second function defined on a domain of
$\gamma_{\t+\pi/2}\times\R$. If $v'\ge v$ on the intersection of their
domains of definition, we say that the $Y_\t$-graph of $v'$ lies on
the positive $Y_\t$-side of the $Y_\t$-graph of $v$.

The same notion can be defined for the vector field $Z_\t$. If $D$ is
a domain in the vertical geodesic plane $\gamma_{\t+\pi/2}\times\R$
and $v$ is a function on $D$ with values in $\R$, the surface
$\{\psi_{v(p)}(p)\ |\ p\in D\}$ is called a $Z_\t$-graph
($(\psi_s)_{s\in\R}$ is the flow of $Z_\t$). This surface is a graph
with respect to $Z_\t$ since it meets each orbit of $Z_\t$ in at most
one point.


\section{The Alexandrov problem for doubly periodic constant mean
  curvature surfaces}
\label{sec:alexandrov}

Let $(\phi_l)_{l\in(-1,1)}$ be the flow of $Y_0$ and consider $G$ the
$\Z^2$ subgroup of $\Isom(\H^2\times\R)$ generated by $\phi_l$ and
$T(h)$, for some positive $l$ and $h$.  We denote by $\M$ the quotient
of $\H^2\times\R$ by $G$. The manifold $\M$ is diffeomorphic to
$\T^2\times\R$. Moreover, $\M$ is foliated by the family of tori
$\T(s)=(d(s)\times\R)/G$, $s\in\R$ (we recall that $d(s)$ is an
equidistant to $\g_0$). All the $\T(s)$ are intrinsically flat and
have constant mean curvature $\tanh(s)/2$; $\T(0)$ is totally
geodesic.

In this section, we study compact embedded constant mean curvature
surfaces in $\M$. The tori $T(s)$ are examples of such surfaces when
$0\leq H<1/2$.

First, let us observe what happens in $(\H^2\times\R)/G'$, where $G'$
is the subgroup generated by $T(h)$. This quotient is isometric to
$\H^2\times\S^1$. Let $\Sigma$ be a compact embedded constant mean
curvature $H$ surface in $\H^2 \times\S^1$.  The surface $\Sigma$
separates $\H^2\times\S^1$.  Indeed, if it is not the case, the exists
a smooth jordan curve whose intersection number with $\Sigma$ is $1$
modulo $2$. In $\H^2\times\S^1$, this jordan curve can be moved such
that it does not intersect $\Sigma$ any more, which is impossible
since the intersection number modulo $2$ is invariant by homotopy.

Now, we consider $\gamma$ a geodesic in $\H^2$ and $(\ell_s)_{s\in\R}$
the family of geodesics in $\H^2$ orthogonal to $\gamma$ that foliates
$\H^2$. By the maximum principle using the vertical annuli
$\ell_s\times \S^1$, we get that $H>0$, since $\Sigma$ is compact. We
can apply the standard Alexandrov reflection technique with respect to
the family $(\ell_s\times \S^1)_{s\in\R}$. We obtain that $\Sigma$ is
symmetric with respect to some $\ell_{s_0}\times \S^1$. Doing this for
every $\gamma$, one proves that $\Sigma$ is a rotational surface
around a vertical axis $\{p\}\times\S^1$ ($p\in\H^2$). $\Sigma$ is
then either a constant mean curvature sphere coming from the spheres
of $\H^2\times\R$ or the quotient by $G'$ of a vertical cylinder or
unduloid of axis $\{p\}\times\R$. This proves that, necessarily,
$H>1/2$. These surfaces are the only ones in $\H^2\times\S^1$ which
have a compact projection on $\H^2$. In $\H^2\times\R$, determining
which properly embedded CMC surfaces have a compact projection on
$\H^2$ (\textit{i.e.} is included in a vertical cylinder) is an open
question.

The spheres, the cylinders and the unduloids can also be quotiented by
$G$, if they are well placed in $\H^2\times\R$ with respect to
$\g_0\times\R$.  They give examples of compact embedded CMC surfaces
in $\M$ for $H>1/2$.

We remark that the vector field $Z_{\pi/2}$ is invariant by the group
$G$, so it is well defined in $\M$. Moreover its integral curves are
the geodesics orthogonal to $\T(0)$. This implies that the notion of
$Z_{\pi/2}$ graph is well defined in $\M$. We have the following
answer to the Alexandrov problem in $\M$.

\begin{theorem}\label{thm:alexandrov}
  Let $\Sigma\subset \M$ be a compact constant mean curvature embedded
  surface. Then, $\Sigma$ is either:
  \begin{enumerate}
  \item a torus $\T(s)$, for some $s$; or
  \item a ``rotational" sphere; or
  \item the quotient of a vertical unduloid (in particular, a vertical
    cylinder over a circle); or
  \item a $Z_{\pi/2}$-bigraph with respect to $\T(0)$.
  \end{enumerate}
  Moreover, if $\Sigma$ is minimal, then $\Sigma=\T(0)$.
\end{theorem}

The first thing we have to remark is that the last item can occur. Let
$\g$ be a compact geodesic in the totally geodesic torus $\T(0)$. From
a result by R.~Mazzeo and F.~Pacard \cite{MaPa}, we know that there
exist embedded constant mean curvature tubes that partially foliate a
tubular neighborhood of $\g$. So if $\g$ is not vertical, these
constant mean curvature surfaces can not be of one of the three first
type. If fact, these surfaces can be also directly derived
from~\cite{SaE} (see also~\cite{Onn}).  They have mean curvature
larger than $1/2$.

The second remark is that we do not know if there exist constant mean
curvature $1/2$ examples. If they exist, they are of the fourth type.

Very recently, J.M. Manzano and F. Torralbo~\cite{MaTo} construct, for
each value of $H>1/2$, a 1-parameter family of ``horizontal
unduloidal-type surfaces'' in $\H^2\times\R$ of bounded height which
are invariant by a fixed $\phi_l$. They conjecture that all these
examples are embedded.  The limit surfaces in the boundary of this
family are a rotational sphere and a horizontal cylinder.

\begin{proof}
  Let $\Sigma$ be a compact embedded constant mean curvature surface
  in $\M$ and consider a connected component $\widetilde\Sigma$ of its
  lift to $\H^2\times\S^1$. If $\widetilde\Sigma$ is compact, the
  above study proves that we are then in cases \textit{2} or
  \textit{3}.  We then assume that $\widetilde\Sigma$ is not compact.
  Even if $\widetilde\Sigma$ is not compact, the same argument as
  above proves that it separates $\H^2\times\S^1$ into two connected
  components. We also assume that $\widetilde\Sigma\neq\g_0\times\S^1$
  (otherwise we are in Case \textit{1}). Then, up to a reflection
  symmetry with respect to $\g_0\times\S^1$, we can assume that
  $\widetilde\Sigma\cap (\{x\ge0\}\times\S^1)$ is non empty.

  Let $\g$ be an integral curve of $Z_{\pi/2}$, \textit{i.e.} a
  geodesic orthogonal to $\g_0\times\S^1$. We denote by $P(s)$ the
  totally geodesic vertical annulus of $\H^2\times\S^1$ which is
  normal to $\g$ and tangent to $d(s)\times\S^1$. Since
  $\widetilde\Sigma$ is a lift of the compact surface $\Sigma$,
  $\widetilde\Sigma$ stays at a finite distance from
  $\g_0\times\S^1$. Far from $\g$, the distance from $P(s)$ to
  $\g_0\times\S^1=P(0)$ tends to $+\infty$, if $s\neq0$. Thus
  $P(s)\cap \widetilde\Sigma$ is compact for $s\neq 0$, and it is
  empty if $|s|$ is large enough. So start with $s$ close to $+\infty$
  and let $s$ decrease until a first contact point between
  $\widetilde\Sigma$ and $P(s)$, for $s=s_0>0$.  If $\widetilde\Sigma$
  is minimal, by the maximum principle we get
  $\widetilde\Sigma=P(s_0)$. But the quotient of $P(s_0)$ is not
  compact in $\M$. We then deduce that $\widetilde\Sigma$ is not
  minimal. This proves that the only compact embedded minimal surface
  in $\M$ is $\T(0)$.

  By the maximum principle, we know that the (non-zero) mean curvature
  vector of $\widetilde{\Sigma}$ does not point into $\cup_{s\ge
    s_0}P(s)$.  Let us continue decreasing $s$ and start the
  Alexandrov reflection procedure for $\widetilde\Sigma$ and the
  family of vertical totally geodesic annuli $P(s)$. Suppose there is
  a first contact point between the reflected part of
  $\widetilde\Sigma$ and $\widetilde\Sigma$, for some $s_1>0$. Then
  $\widetilde\Sigma$ is symmetric with respect to $P(s_1)$. Since
  $s_1>0$, then $\widetilde\Sigma\cap(\cup_{s_1\leq s\leq s_0}P(s))$
  is compact. We get that $\widetilde\Sigma$ is compact, a
  contradiction. Hence we can continue the Alexandrov reflection
  procedure until $s=0$ without a first contact point.  This implies
  that $\widetilde\Sigma \cap (\{x\ge0\}\times\S^1)$ is a Killing
  graph above $\g_0\times\S^1$, for the Killing vector field $Y$
  corresponding to translations along $\g$ (we notice that, along
  $\g$, $Y$ and $Z_{\pi/2}$ coincide).  Hence $\g$ has at most one
  intersection point $p$ with $\widetilde\Sigma \cap
  (\{x\ge0\}\times\S^1)$ and this intersection is transverse.

  Since at the first contact point between $\widetilde\Sigma$ and
  $P(s)$ (for $s=s_0$) the mean curvature vector of $\widetilde\Sigma$
  does not point into $\cup_{s\ge s_0}P(s)$, we have that, for any
  $s'\in(0,s_0]$, the mean curvature vector of $\widetilde\Sigma$ on
  $\widetilde\Sigma\cap P(s')$ does not point into $\cup_{s\ge
    s'}P(s)$. In particular, the mean curvature vector of
  $\widetilde\Sigma$ at $p$ points to the opposite direction as
  $\Z_{\pi/2}$.  Doing this for every geodesic $\g$ orthogonal to
  $\g_0\times\S^1$, we get that $\widetilde\Sigma \cap
  (\{x\ge0\}\times\S^1)$ is a $Z_{\pi/2}$ graph.

  Now let us suppose that $\widetilde\Sigma$ is included in
  $\{x\ge0\}\times\S^1$, and let $s_2\ge 0$ and $s_3>0$ be the minimum
  and the maximum of the distance from $\widetilde\Sigma$ to
  $\g_0\times\S^1$, respectively. Thus $\widetilde\Sigma$ is contained
  between $d(s_2)\times\S^1$ and $d(s_3)\times\S^1$. Because of the
  orientation of the mean curvature vector at the contact points of
  $\widetilde\Sigma$ with $d(s_2)\times\S^1$ and $d(s_3)\times\S^1$,
  we get
  \[
  H_{d(s_2)\times\S^1}\ge H_{\widetilde\Sigma}\ge
  H_{d(s_3)\times\S^1}.
  \]
  But $H_{d(s_2)\times\S^1}\le H_{d(s_3)\times\S^1}$, hence $s_2=s_3$
  and $\widetilde\Sigma=d(s_2)\times\S^1$. This is, we are in Case
  \textit{1}.

  Then we assume that $\widetilde\Sigma\cap (\{x< 0\}\times\S^1)$ is
  non empty. Using the totally geodesic vertical annuli $P(s)$ for
  $s\le 0$, we prove as above that $\widetilde\Sigma \cap
  (\{x\le0\}\times\S^1)$ is a $Z_{\pi/2}$ graph. Moreover the mean
  curvature vector points in the same direction as $Z_{\pi/2}$. This
  implies that $\widetilde\Sigma$ is normal to $\g_0\times\S^1$. Thus,
  in the Alexandrov reflection procedure, a first contact point
  between the reflected part of $\widetilde\Sigma$ and
  $\widetilde\Sigma$ occurs for $s=0$. $\widetilde\Sigma$ is then
  symmetric with respect to $P(0)=\g_0\times\S^1$: we are in Case
  \textit{4}.
\end{proof}


\section{Minimal surfaces invariant by a $\Z$
  subgroup}\label{sec:singly}

In this section, we are interested in constructing minimal surfaces
which are invariant by a $\Z$ subgroup of $\Isom(\H^2\times\R)$.  At
this time, only few non-trivial singly periodic examples are known:
There are examples invariant by a vertical
translation~\cite{NeRo,Hau,MoRo}; invariant by a vertical screw
motion~\cite{NeRo,SaTo1}; invariant by a horizontal hyperbolic
translation~\cite{SaTo2,MaRoRo,Pyo}; or invariant by a horizontal
parabolic or a hyperbolic screw motion~\cite{SaE,Onn}.

The subgroups we consider are those generated by a translation
$\phi_l$ along a horizontal geodesic or by a vertical translation
$T(h)$ along $\partial_t$. The surfaces we construct are similar to
Scherk's singly periodic minimal surfaces and Karcher's Saddle Towers
of $\R^3$.


\subsection{Horizontal singly periodic Scherk minimal surfaces}
\label{subsec:horizontal}

In this subsection we construct a 1-parameter family of minimal
surfaces in $\H^2\times\R$, called "horizontal singly periodic Scherk
minimal surfaces".  Each of these surfaces can be seen as the
desingularization of the intersection of a vertical geodesic plane and
the horizontal slice $\H^2\times\{0\}$, and it is invariant by a
horizontal hyperbolic translation along the geodesic of intersection.

We fix $\mu\in(0,1)$ and define $q_\mu=(0,\mu)$ and
$q_{-\mu}=(0,-\mu)$.  Given $R>0$, we denote by $\Ome(R)$ the compact
domain in $\{x\ge 0\}$ between $E(R)$ and the geodesic lines $g(\mu)$,
$g(-\mu)$, $\g_0$, where $E(R)$ is the arc contained in the
equidistant line $d(R)$ which goes from $g(\mu)$ to $g(-\mu)$, see
Figure~\ref{fig:omega}. Let $u_R$ be the solution to \eqref{mse} over
$\Ome(R)$ with boundary values zero on $\partial\Ome(R)\setminus \g_0$
and value $R$ on $\overline{q_\mu q_{-\mu}}$ (minus its endpoints).
By the maximum principle, $u_{R'}>u_R$ on $\Omega_R$, for any $R'>R$.

\begin{figure}
  \begin{center}
    \includegraphics[width=0.5\textwidth]{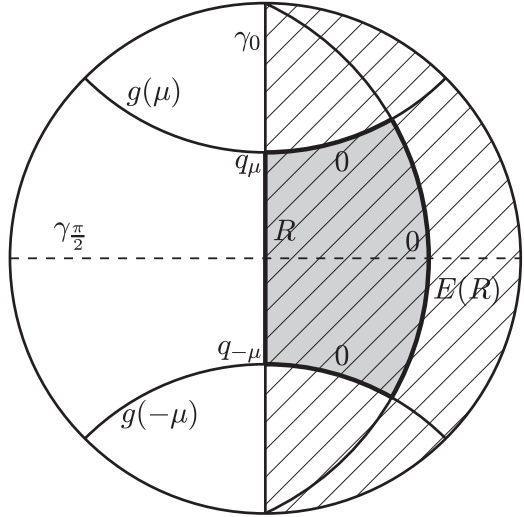}
  \end{center}
  \caption{The shadowed domain is $\Ome(R)$, with the prescribed
    boundary data. The ruled region corresponds to $D^+$.}
  \label{fig:omega}
\end{figure}

Let us denote $D^+=\{x\ge 0\}$ the hyperbolic halfplane bounded by
$\g_0$. On $D^+$, we consider the solution $v$ of \eqref{mse}
discovered by U.~Abresch and R. Sa Earp, which takes value $+\infty$
on $\g_0$ and $0$ on the asymptotic boundary $\partial_\infty D^+$
(see Appendix~\ref{app:barrier}).  Such a $v$ is a barrier from above
for our construction, since we have $u_R\le v$ for any $R$.

Since $(u_R)_R$ is a monotone increasing family bounded from above by
$v$, we get that $u_R$ converges as $R\rightarrow +\infty$ to a
solution $u$ of \eqref{mse} on $\Ome(\infty)=\cup_{R>0}\Ome(R)$, with
boundary values $+\infty$ over $\overline{q_\mu q_{-\mu}}$ (minus its
endpoints) and $0$ over the remaining boundary (including the
asymptotic boundary $E(\infty)$ at infinity).  In fact, this solution
$u$, which is unique, can be directly derived from Theorem~4.9
in~\cite{MaRoRo}.

\begin{figure}
  \begin{center}
    \includegraphics[width=0.57\textwidth]{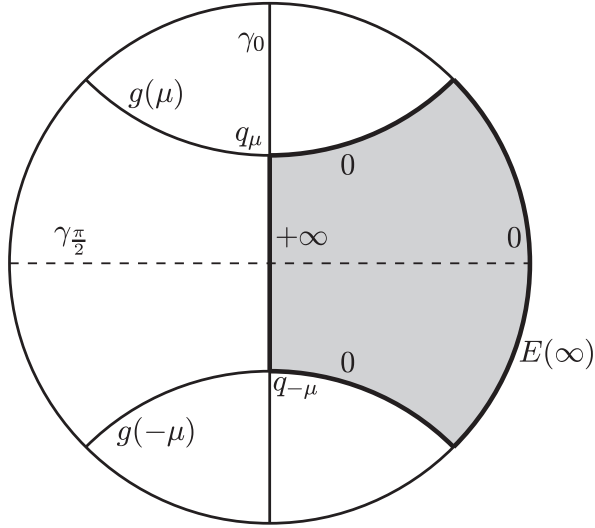}
  \end{center}
  \caption{The domain $\Ome(\infty)$ with the prescribed boundary
    data.}
  \label{fig:omega_inf}
\end{figure}

Let $\Sigma_R$ be the minimal graph of $u_R$.  $\Sigma_R$ is in fact
the solution to a Plateau problem in $\H^2\times\R$ whose boundary is
composed of horizontal and vertical geodesic arcs and the arc
$E(R)\times\{0\}$.  Let $\phi_l$ denote the flow of $Y_0$. Using the
foliation of $\H^2\times\R$ by the vertical planes
$\phi_l(\gamma_{\pi/2}\times\R)=g(l)\times\R$, $l\in(-1,1)$, the
Alexandrov reflection technique proves that $\Sigma_R$ is a
$Y_0$-bigraph with respect to $\g_{\pi/2}\times\R$. So
$\Sigma_R^+=\Sigma_R\cap\{y\ge 0\}$ is a $Y_0$-graph. Thus, the same
is true for the minimal graph $\Sigma$ of $u$ and for
$\Sigma^+=\Sigma\cap\{y\ge 0\}$.

The boundary of $\Sigma$ is composed of the vertical half-lines
$\{q_\mu\} \times\R^+$, $\{q_{-\mu}\}\times\R^+$ and the two halves
$g^+(\mu),g^+(-\mu)$ of the horizontal geodesics $g(\mu),g(-\mu)$.
The expected ``horizontal singly periodic Scherk minimal surface'' is
obtained by rotating recursively $\Sigma$ an angle $\pi$ about the
vertical and horizontal geodesics in its boundary.  This ``horizontal
singly periodic Scherk minimal surface'' is properly embedded,
invariant by the horizontal translation $\phi_\mu^4$ along $\g_0$ and,
far from $\g_0\times\{0\}$, it looks like $(\g_0\times\R)\cup
\{t=0\}$.

\begin{proposition}\label{prop:Scherk1p}
  For any $\mu\in(0,1)$, there exists a properly embedded minimal
  surface ${\cal M}_\mu$ in $\H^2\times\R$ invariant by the horizontal
  hyperbolic translation $\phi_\mu^4$ along $\g_0$, that we call {\em
    horizontal singly periodic Scherk minimal surface}. In the
  quotient by $\phi_\mu^4$, ${\cal M}_\mu$ is topologically a sphere
  minus four points corresponding its ends: it has one top end
  asymptotic to $(\g_0\times\R^+)/\phi_\mu^4$, one bottom end
  asymptotic to $(\g_0\times\R^-)/\phi_\mu^4$, one left end asymptotic
  to $\{t=0,x<0\}/\phi_\mu^4$, and one right end asymptotic to
  $\{t=0,x>0\}/\phi_\mu^4$. Moreover, ${\cal M}_\mu/\phi_\mu^4$
  contains the vertical lines $\{q_{\pm\mu}\}\times\R$ and the
  horizontal geodesics $g(\pm\mu)\times\{0\}$, and it is invariant by
  reflection symmetry with respect to the vertical geodesic plane
  $\g_{\pi/2}\times\R$.
\end{proposition}

\begin{remark}
  \textbf{``Generalized horizontal singly periodic Scherk minimal
    surfaces''.}

  Consider the domain $\Omega(\infty)$ with prescribed boundary data
  $+\infty$ on $\overline{q_\mu q_{-\mu}}$, $0$ on $g^+(\mu)\cup
  g^+(-\mu)$ and a continuous function $f$ on the asymptotic boundary
  $E(\infty)$ of $\Omega(\infty)$ at infinity. By Theorem~4.9
  in~\cite{MaRoRo}, we know there exists a (unique) solution to this
  Dirichlet problem associated to equation~\eqref{mse}.

  By rotating recursively such a graph surface an angle $\pi$ about
  the vertical and horizontal geodesics in its boundary, we get a
  ``generalized horizontal singly periodic Scherk minimal surface''
  ${\cal M}_\mu(f)$, which is properly embedded and invariant by the
  horizontal translation $\phi_\mu^4$ along $\g_0$.  Such a ${\cal
    M}_\mu(f)$ can be seen as the desingularization of the vertical
  geodesic plane $\g_0\times\R$ and a periodic minimal entire graph
  invariant by the horizontal translation $\phi_\mu^4$ along~$\g_0$.
  Moreover, the surface ${\cal M}_\mu(f)$ contains the vertical lines
  $\{q_{\pm\mu}\}\times\R$ and the horizontal geodesics
  $g(\pm\mu)\times\{0\}$.

  In general, ${\cal M}_\mu(f)$ contains vertical geodesic arcs at the
  infinite boundary $\partial_\infty\H^2\times\R$ (over the endpoints
  of $g(\pm\mu)$ and their translated copies). To avoid such vertical
  segments, we take $f$ vanishing on the endpoints of $E(\infty)$.
\end{remark}


\subsection{A Plateau construction of vertical Saddle
  Towers}\label{subsec:vertical}

In this section, we construct the 1-parameter family of most symmetric
vertical Saddle Towers in $\H^2\times\R$, which can be seen as the
desingularization of $n$ vertical planes meeting at a common axis with
angle $\t=\pi/n$, for some $n\ge 2$. When $n=2$, the corresponding
examples are usually called ``vertical singly periodic Scherk minimal
surfaces''. For any fixed $n\geq 2$, these examples are included in
the $(2n-3)$-parameter family of vertical Saddle Towers constructed by
Morabito and the second author in~\cite{MoRo}. These surfaces are all
invariant by a vertical translation $T(h)$.

A fundamental piece of the Saddle Tower we want to construct is
obtained by solving a Plateau problem. We now consider a more general
Plateau problem, that will be also used in Sections~\ref{sec:2per}
and~\ref{sec:isomH2}.

Given an integer $n\ge 2$, we fix $\t=\pi/n$.  We consider in $\H^2$
the points
\[
p_\l=(\l\sin\t,\l\cos\t)\quad \mbox{ and }\quad q_\mu=(0,\mu),
\]
for any $\l\in(0,1]$ and any $\mu\in(0,1]$ (see
Figure~\ref{fig:plateau1}). Given $h>0$, we call
$W_{h,\l,\mu}\subset\H^2\times\R$ the triangular prism whose top and
bottom faces are two geodesic triangular domains at heights $0$ and
$h$: the bottom triangle has vertices $(p_\l,0),(\0,0),(q_{\mu},0)$
and the top triangle is its vertical translation to height $h$.

\begin{figure}
  \begin{center}
    \includegraphics[width=0.5\textwidth]{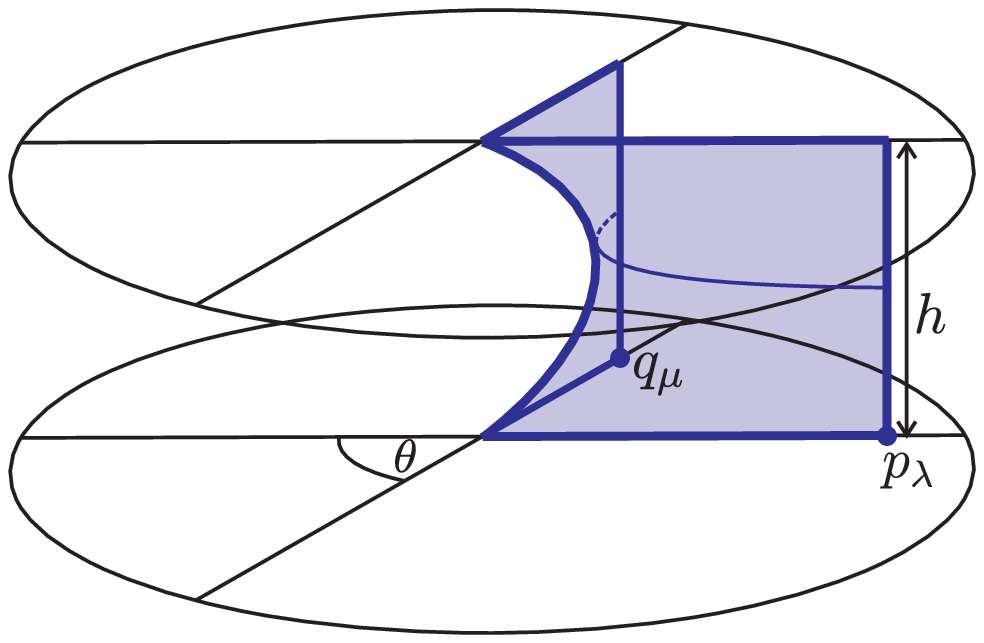}
  \end{center}
  \caption{The embedded minimal disk $\Sigma_{h,\l,\mu}$ bounded by
    $\G_{h,\l,\mu}$.}
  \label{fig:plateau1}
\end{figure}

If $\l<1$ and $\mu<1$, we consider the following Jordan curve in the
boundary of $W_{h,\l,\mu}$:

\[
\begin{array}{ll}
  \G_{h,\l,\mu} = & \overline{(q_{\mu},0)\,(\0,0)}\cup
  \overline{(\0,0)\,(p_\l,0)}\cup
  \overline{(p_\l,0)\,(p_\l,h)} \\
  & \cup
  \overline{(p_\l,h)\,(\0,h)}\cup
  \overline{(\0,h)\,(q_{\mu},h)}\cup
  \overline{(q_{\mu},h)\,(q_{\mu},0)}
\end{array}
\]
(see Figure~\ref{fig:plateau1}). Since $\partial W_{h,\l,\mu}$ is
mean-convex and $\G_{h,\l,\mu}$ is contractible in $W_{h,\l,\mu}$,
there exists an embedded minimal disk $\Sigma_{h,\l,\mu}\subset
W_{h,\l,\mu}$ whose boundary is $\G_{h,\l,\mu}$ (see Meeks and
Yau~\cite{MeYa}).

\begin{claim}\label{cl:uniqueness}
  $\Sigma_{h,\l,\mu}$ is the only compact minimal surface in
  $\H^2\times\R$ bounded by $\G_{h,\l,\mu}$. Moreover,
  $\Sigma_{h,\l,\mu}$ is a minimal $Y_{\t/2}$-graph and it lies on the
  positive $Y_{\t/2}$-side of $\Sigma_{h,\l',\mu'}$, for any
  $\l'\le\l$ and any $\mu'\le\mu$.
\end{claim}

\begin{proof}
  Let $\Sigma,\Sigma'\subset\H^2\times\R$ be two compact minimal
  surfaces with $\partial\Sigma=\G_{h,\l,\mu}$ and
  $\partial\Sigma'=\G_{h,\l',\mu'}$, where $\l'\le \l$ and $\mu'\le
  \mu$. First observe that, by the convex hull property (or by the
  maximum principle using vertical geodesic planes and horizontal
  slices), $\Sigma\subset W_{h,\l,\mu}$ and $\Sigma'\subset
  W_{h,\l',\mu'}$.

  Let $(\phi_l)_{l\in(-1,1)}$ be the flow of $Y_{\t/2}$. For $l$ close
  to $-1$, $\phi_l(W_{h,\l',\mu'})\cap W_{h,\l,\mu}=\emptyset$ and,
  for $-1<l<0$, $\phi_l(\G_{h,\l',\mu'})$ and $W_{h,\l,\mu}$ do not
  intersect. So letting $l$ increase from $-1$ to $0$, we get by the
  maximum principle that $\phi_l(\Sigma')$ and $\Sigma$ do not
  intersect until $l=0$. When $\l=\l'$ and $\mu=\mu'$, this implies
  that $\Sigma=\Sigma'$ (hence $\Sigma = \Sigma_{h,\l,\mu}$) and it is
  a minimal $Y_{\t/2}$-graph.  Also this translation argument shows
  that $\Sigma$ lies on the positive $Y_{\t/2}$-side of $\Sigma'$ when
  $\l'<\l$ and $\mu'<\mu$.
\end{proof}

From Claim~\ref{cl:uniqueness}, we deduce the continuity of $\Sigma_
{h,\l,\mu}$ in the $\l$ and $\mu$ parameters. The surfaces $\Sigma_
{h,\l,\mu}$ will be used in Sections~\ref{sec:2per} and
\ref{sec:isomH2} for the construction of doubly periodic minimal
surfaces and surfaces invariant by a subgroup of $\Isom(\H^2)$. More
precisely, in the following subsection we construct surfaces from
$\Sigma_{h,\l,\mu}$ that we use in the sequel.

\medskip

Now we only consider the $\l=\mu$ case.  As $Y_{\t/2}$-graphs, the
surfaces $\Sigma_{h,\mu,\mu}$ form an increasing family in the $\mu$
parameter.  So if we construct a ``barrier from above'', we could
ensure the convergence of $\Sigma_{h,\mu,\mu}$ when $\mu\to 1$.

\begin{figure}
  \begin{center}
    \includegraphics[width=0.5\textwidth]{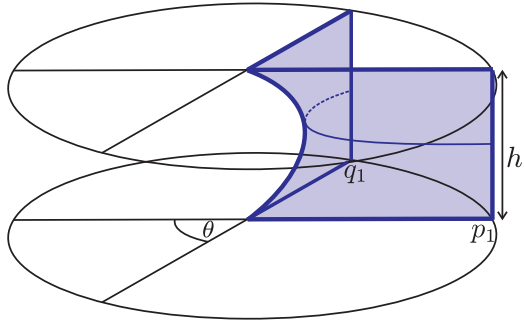}
  \end{center}
  \caption{The embedded minimal disk $\Sigma_{h}$ bounded by
    $\G_{h}$.}
  \label{fig:plateau2}
\end{figure}

On the ideal triangular domain of vertices $\0,p_1,q_1$, there exists
a solution $u$ to the vertical minimal graph equation~\eqref{mse}
which takes boundary values $0$ on $\overline{q_1\0}$ and
$\overline{\0 p_1}$ and $+\infty$ on $\overline{p_1 q_1}$. Let $S_0$
and $S_h$ be, respectively, the graph surfaces of $u$ and $h-u$.

Using the same argument as in Claim~\ref{cl:uniqueness}, we conclude
that both $S_0$ and $S_h$ are $Y_{\t/2}$-graphs and lie on the
positive $Y_{\t/2}$-side of $\Sigma_{h,\mu,\mu}$, for any $\mu$. They
are the expected ``barriers from above''.

Using the monotonicity and the barriers, we conclude that there exists
a limit $\Sigma_h$ of the minimal $Y_{\t/2}$-graphs
$\Sigma_{h,\mu,\mu}$ when $\mu\to 1$. And it is also a minimal
$Y_{\t/2}$-graph. The surface $\Sigma_h$ is a minimal disk bounded by
\[
\G_h =\overline{(q_1,0)(\0,0)} \cup \overline{(\0,0)(p_1,0)} \cup
\overline{(q_1,h)(\0,h)} \cup \overline{(\0,h)(p_1,h)}.
\]
In fact, applying the techniques of Claim~\ref{cl:uniqueness}, we get
that $\Sigma_h$ is the only minimal disk of $\H^2\times\R$ bounded by
$\G_h$ which is contained in $W_{h,1,1}$. By uniqueness, $\Sigma_h$ is
symmetric with respect to the vertical plane $\gamma_{\t/2}\times\R$
and the horizontal slice $\H^2\times\{h/2\}$.

Now we can extend $\Sigma_h$ by doing recursive symmetries along the
horizontal geodesics in its boundary. The surface we obtain is
properly embedded, invariant by the vertical translation $T(2h)$ and
asymptotic to the $n$ vertical planes $\g_{k\t}\times\R$, $0\leq k\leq
n-1$, outside of a large vertical cylinder with axis $\{\0\}\times\R$.

\begin{proposition}
  For any natural $n\geq 2$ and any $h>0$, there exists a properly
  embedded minimal surface ${\cal M}_h(n)$ in $\H^2\times\R$ invariant
  by the vertical translation $T(2h)$ and asymptotic to the $n$
  vertical planes $\g_{\frac{k\pi}{n}}\times\R$, for $0\leq k\leq
  n-1$, far from $\{\0\}\times\R$. Moreover, ${\cal M}_h(n)$ contains
  the horizontal geodesics $\g_{\frac{k\pi}{n}}\times\{0\}$, $0\leq
  k\leq n-1$, and is invariant by reflection symmetry with respect to
  the vertical geodesic planes $\g_{(\frac 1 2
    +k)\frac{\pi}{n}}\times\R$, with $0\leq k\leq n-1$, and respect to
  the horizontal slices $\H^2\times\{\pm h/2\}$. We call such a
  surface {\em (most symmetric) vertical Saddle Tower}.
\end{proposition}


\subsection{The minimal surfaces $\Sigma_{h,\l}$ and
  $M_{h,\l}$}\label{subsec:M}

In order to prepare our work in Sections~\ref{sec:2per} and
\ref{sec:isomH2}, we continue to study the solutions of the Plateau
problem introduced in Subsection~\ref{subsec:vertical}.

Recall that $n\geq 2$ is an integer number,
$\t=\pi/n$ and $\l,\mu\in(0,1)$. We now fix $\l$ and $h>0$, and we
consider the family of $Y_{\t/2}$-graphs $\Sigma_{h,\l,\mu}$ as $\mu$
moves. This family is monotone increasing in the $\mu$-parameter. And,
for fixed $h$, the $Y_{\t/2}$-graphs $\Sigma_{h,\l,\mu}$ are bounded
from above by the surface $\Sigma_h$ constructed in the preceding
subsection. Thus $\Sigma_{h,\l,\mu}$ converges to a minimal
$Y_{\t/2}$-graph $\Sigma_{h,\l}$ when $\mu\rightarrow 1$. This surface
is an embedded minimal disk bounded by
\[
\G_{h,\l} =
\begin{array}{l} \overline{(q_1,0)\,(\0,0)}\cup
  \overline{(\0,0)\,(p_\l,0)}\cup \overline{(p_\l,0)\,(p_\l,h)}\\
  \cup \overline{(p_\l,h)\,(\0,h)}\cup \overline{(\0,h)\,(q_1,h)}.
\end{array}
\]
In fact, applying the techniques of Claim~\ref{cl:uniqueness}, we
conclude that $\Sigma_{h,\l}$ is the only minimal disk contained in
$W_{h,\l,1}$ which is bounded by $\G_{h,\l}$.

The Alexandrov reflection method with respect to horizontal slices
says that every $\Sigma_{h,\l,\mu}$ is a symmetric vertical bigraph
with respect to $\H^2\times\{h/2\}$ (see
Appendix~\ref{app:alexandrov}). Hence this is also true for
$\Sigma_{h,\l}$.

\begin{figure}
  \begin{center}
    \includegraphics[width=0.5\textwidth]{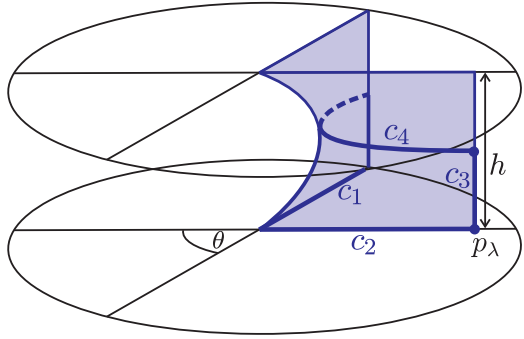}
  \end{center}
  \caption{The minimal disk $\Sigma_{h,\l}$ bounded by $\G_{h,\l}$,
    and the minimal vertical graph $M_{h,\l}=\Sigma_{h,\l}\cap\{0\leq
    t\leq h/2\}$ bounded by $c_1\cup c_2\cup c_3\cup c_4$.}
  \label{fig:plateau3}
\end{figure}

We consider
\[
M_{h,\l}= \Sigma_{h,\l}\cap\{0\leq t\leq h/2\},
\]
which is a minimal vertical graph bounded by $c_1,c_2,c_3,c_4$ (see
Figure~\ref{fig:plateau3}), where:
\begin{itemize}
\item $c_1= \overline{(q_1,0)\,({\bf 0},0)}=\g_0^+$ is half a complete
  horizontal geodesic line;
\item $c_2=\overline{({\bf 0},0)\,(p_\l,0)}$ is a horizontal geodesic
  of length $\ln\left(\frac{1+\l}{1-\l}\right)$, forming an angle $\t$
  with $c_1$ at $A_0=({\bf 0},0)$;
\item $c_3=\overline{(p_\l,0)\,(p_\l,h/2)}$ is a vertical geodesic
  line of length $h/2$;
\item $c_4=M_{h,\l}\cap\{t=h/2\}$ is a horizontal geodesic curvature
  line with endpoints $(p_\l,h/2)$ and $(q_1,h/2)$.
\end{itemize}

The domain $\Omega_0$ over which $M_{h,\l}$ is a graph is included in
the triangular domain of vertices $\0, p_\l, q_1$, and it is bounded
by $\overline{q_1 \0}$, $\overline{\0 p_\l}$ and $\pi(c_4)$.  The
latter curve goes from $p_\l$ to $q_1$ and is concave with respect to
$\Omega_0$ because of the boundary maximum principle using vertical
geodesic planes, which implies that the mean curvature vector of
$\pi(c_4)\times\R$ points outside $\Omega_0\times \R$.

On $M_{h,\l}$, we fix the unit normal vector field $N$ whose
associated angle function $\nu=\langle N,\partial_t\rangle$ is
non-negative. The vector field $N$ extends smoothly to $\partial
M_{h,\l}$ (by Schwarz symmetries). It is not hard to see that $\nu$
only vanishes on $c_3\cup c_4$, and $\nu=1$ at $A_0=(\0,0)$.

Since $\Sigma_{h,\l}$ is a $Y_{\t/2}$-graph, then it is stable, so it
satisfies a curvature estimate away from its boundary. Hence the
curvature is uniformly bounded on $M_{h,\l}$ away from $c_1$, $c_2$
and $c_3$. Besides, $M_{h,\l}$ can be extended by symmetry along $c_1$
and $c_2$ as a vertical graph, thus as a stable surface. Hence, on
$M_{h,\l}$, the curvature is uniformly bounded away from $c_3$.

Because of this curvature estimate and since $M_{h,\l}\subset
W_{h,\l,1}$, the angle function $\nu$ goes to zero as we approach
$q_1\times[0,h/2]$, and the asymptotic intrinsic distance from $c_1$
to $c_4$ is $h/2$.


\section{Doubly periodic minimal surfaces}\label{sec:2per}
In this section, we construct doubly periodic minimal surfaces,
\textit{i.e.} properly embedded minimal surfaces invariant by a
subgroup of $\Isom(\H^2\times\R)$ isomorphic to $\Z^2$. In fact, we
only consider subgroups generated by a hyperbolic translation along a
horizontal geodesic and a vertical translation. More precisely, let
$(\phi_l)_{l\in(-1,1)}$ be the flow of $Y_0$. We are interested in
properly embedded minimal surfaces which are invariant by the subgroup
of $\Isom(\H^2\times\R)$ generated by $\phi_l$ and $T(h)$, for
fixed~$l$ and~$h$.  We notice that the quotient $\M$ of $\H^2\times\R$
by this subgroup is diffeomorphic to $\T\times\R$, where $\T$ is a
2-torus.

One trivial example of a doubly periodic minimal surface is the
vertical plane $\gamma_0\times\R$. The quotient surface is
topologically a torus and it is in fact the only compact minimal
surface in the quotient (see Theorem~\ref{thm:alexandrov}). Other
trivial examples are given by the quotients of a horizontal slice
$\H^2\times\{t_0\}$ or a vertical totally geodesic minimal plane
$g(\mu)\times\R$. Both cases give flat annuli in the quotient.

In the following subsections, we construct non-trivial examples, that
are similar to minimal surfaces of $\R^3$ built by H.~Karcher
in~\cite{Kar}. Their ends are asymptotic to the horizontal and/or the
vertical flat annuli described above.


\subsection{Doubly periodic Scherk minimal surfaces}
\label{subsec:scherk2p}

In this subsection we construct minimal surfaces of genus zero in $\M$
which have two ends asymptotic to two vertical annuli and two ends
asymptotic to two horizontal annuli in the quotient.  These examples
are similar to the doubly periodic Scherk minimal surface in $\R^3$.

\begin{figure}
  \begin{center}
    \includegraphics[width=0.5\textwidth]{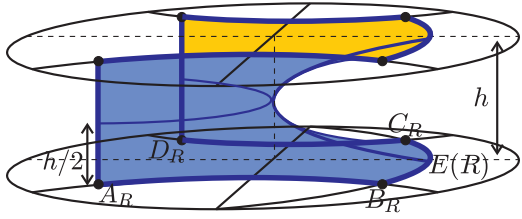}
  \end{center}
  \caption{The Jordan curve $\G_R$ and the embedded minimal disk
    $\Sigma_R$ bounded by $\G_R$.}
   \label{fig:scherk2p}
 \end{figure}

 Let $A_R$, $B_R$ in $g(-\mu)$ and $C_R$, $D_R$ in $g(\mu)$ at
 distance $R$ from $\g_0$ such that $A_R$ and $D_R$ are in $\{x<0\}$
 and $B_R$ and $C_R$ are in $\{x>0\}$, see
 Figure~\ref{fig:scherk2p}. We fix $h>\pi$ and consider the following
 Jordan curve:
 \[
 \begin{array}{ll}
   \G_R = & \overline{(A_R,0)\,(B_R,0)}\cup
   (E(R)\times\{0\})\cup \overline{(C_R,0)\,(D_R,0)}\\
   &  \cup \overline{(D_R,0)\,(D_R,h)}\cup \overline{(D_R,h)\,(C_R,h)}\cup
   (E(R)\times\{h\})\\
   & \cup \overline{(B_R,h)\,(A_R,h)}\cup
   \overline{(A_R,h)\,(A_R,0)},
 \end{array}
 \]
 where $E(R)$ is the subarc of the equidistant $d(R)$ to $\g_0$ that
 joins $B_R$ to $C_R$.  We consider a least area embedded minimal disk
 $\Sigma_R$ with boundary $\G_R$.

 Using the Alexandrov reflection technique with respect to horizontal
 slices, one proves that $\Sigma_R$ is a vertical bigraph with respect
 to $\{t=h/2\}$ (see Appendix~\ref{app:alexandrov}).

 Since $\Sigma$ is area-minimizing, it is stable. This gives uniform
 curvature estimates far from the boundary. Besides $\Sigma_R\cap
 \{0\le t\le h/2\}$ is a vertical graph that can be extended by
 symmetry with respect to $\overline{(A_R,0)\,(B_R,0)}$ to a larger
 vertical graph.  Thus we also obtain uniform curvature estimates in a
 neighborhood of $\overline{(A_R,0)\,(B_R,0)}$. This is also true for
 the three other horizontal geodesic arcs in $\G_R$.

 Let $A_\infty$ and $D_\infty$ be the endpoints of $g(-\mu)$ and
 $g(\mu)$, that are limits of $A_R$ and $D_R$ as $R\to+\infty$. For
 any $R$, $\Sigma_R$ is on the half-space determined by
 $\overline{A_\infty D_\infty}\times\R$ that contains $\G_R$.

 Since $h>\pi$, we can consider the surface $S_h$ described in
 Appendix~\ref{app:barrier}: $S_h\subset \H^2\times(0,h)$ is a
 vertical bigraph with respect to $\{t=h/2\}$ which is invariant by
 translations along $\g_0$ and whose boundary is
 $(\a\times\{0\})\cup\overline{(0,1,0)(0,1,h)}\cup
 (\a\times\{h\})\cup\overline{(0,-1,h)(0,-1,0)}$, where
 $\a=\partial_\infty\H^2\cap\{x>0\}$.  Let $(\chi_l)_{l\in(-1,1)}$ be
 the flow of the Killing vector field $Y_{\pi/2}$. For $l$ close
 to~$1$, $\chi_l(S_h)$ does not meet $\Sigma_R$. Since
 $\overline{(D_R,0)\,(D_R,h)}$ and $\overline{(A_R,h)\,(A_R,0)}$ are
 the only part of $\G_R$ in $\H^2\times(0,h)$, we can let $l$ decrease
 until $l_R<0$, where $\chi_{l_R}(S_h)$ touches $\Sigma_R$ for the
 first time. Actually, there are two first contact points: $(A_R,h/2)$
 and $(D_R,h/2)$.  By the maximum principle, the surface $\Sigma_R$ is
 contained between $\chi_{l_R}(S_h)$ and $\overline{A_\infty
   D_\infty}\times\R$. We notice that $l_R> l_{R'}$, for any $R'>R$,
 and $l_R\rightarrow l_\infty >-1$, where
 $\chi_{l_\infty}(\g_0)=\overline{A_\infty D_\infty}$.

 We recall that $Z_{\pi/2}$ is the unit vector field normal to the
 equidistant surfaces to $\g_0\times\R$.
 \begin{claim}\label{claim:Zgraph}
   $\Sigma_R\setminus\G_R$ is a $Z_{\pi/2}$-graph over the open
   rectangle $\overline{A_0\,D_0}\times(0,h)$ in $\g_0\times\R$.
 \end{claim}
 \begin{proof}
   It is clear that the projection of $\Sigma_R\setminus \G_R$ over
   $\g_0\times\R$ in the direction of $Z_{\pi/2}$ coincides with
   $\overline{A_0\,D_0}\times(0,h)$.  Let us prove that
   $\Sigma_R\setminus \G_R$ is transverse to $Z_{\pi/2}$. Assume that
   $q$ is a point in $\Sigma_R\setminus \G_R$ where $\Sigma_R$ is
   tangent to $Z_{\pi/2}$. Thus there is a minimal surface $P$ given
   by Appendix~\ref{app:barrier} which is invariant by translation
   along $Z_{\pi/2}$, passes through $q$ and is tangent to
   $\Sigma_R$. Near $q$, the intersection $P\cap \Sigma_R$ is composed
   of $2n$ arcs meeting at $q$, with $n\geq 2$.

   By definition of $P$ and $\Gamma_R$, the intersection $P\cap \G_R$
   is composed either by two points, or by one point and one geodesic
   arc of type $\overline{(A_R,0) \,(B_R,0)}$, or by two arcs of type
   $\overline{(A_R,0)\,(B_R,0)}$ and
   $\overline{(D_R,h)\,(C_R,h)}$. Since $\Sigma_R$ is a disk, we get
   that there exists a component of $\Sigma_R\setminus P$ which has
   all its boundary in $P$. This is impossible by the maximum
   principle, since $\H^2\times\R$ can be foliated by translated
   copies of $P$. The surface $\Sigma_R$ is then transverse to
   $Z_{\pi/2}$.

   Now let $q$ be a point in $\overline{A_0\,D_0}\times(0,h)$, and
   $\ell_q$ be the geodesic passing by $q$ and generated by
   $Z_{\pi/2}$. The intersection of $\ell_q$ with $\Sigma_R$ is always
   transverse, so the number of intersection points does not depend on
   $q$. For $q=(A_0,h/2)$, this number is $1$. Therefore,
   $\Sigma_R\setminus\G_R$ is a $Z_{\pi/2}$-graph over the open
   rectangle $\overline{A_0\,D_0}\times(0,h)$.
 \end{proof}

 Now let $R$ tend to $\infty$. Because of the curvature estimates, and
 using that each $\Sigma_R$ is a $Z_{\pi/2}$-graph bounded by
 $\chi_{l_R}(S_h)$ and $\overline{A_\infty D_\infty}\times\R$, we
 obtain that, the surfaces $\Sigma_R$ converge to a minimal surface
 $\Sigma_\infty$ satisfying the following properties:
   \begin{itemize}
   \item $\Sigma_\infty$ lies in the region of $\{0\le t\le h\}$
     bounded by $g(-\mu)\times\R$, $g(\mu)\times\R$,
     $\overline{A_\infty D_\infty}\times\R$ and
     $\chi_{l_\infty}(S_h)$;
   \item $\partial \Sigma_\infty= (g(-\mu)\times\{0\})\cup
     (g(\mu)\times\{0\})\cup (g(\mu)\times\{h\})\cup
     (g(-\mu)\times\{h\})$;
   \item $\Sigma_\infty\setminus\partial \Sigma_\infty$ is a vertical
     bigraph with respect to $\{t= h/2\}$ and a $Z_{\pi/2}$-graph over
     $\overline{A_0\,D_0}\times(0,h)$;
   \item $\Sigma_\infty\cap \{x\le 0\}$ is asymptotic to
     $g(-\mu)\times[0,h]$ and $g(\mu)\times[0,h]$; and
     $\Sigma_\infty\cap \{x\ge 0\}$ is asymptotic to
     $\{t=0\}/\phi_\mu^2$ and $\{t=h\}/\phi_\mu^2$.
\end{itemize}

After extending $\Sigma_\infty$ by successive symmetries with respect
to the horizontal geodesics contained in its boundary, we obtain a
surface $\Sigma$ invariant by the subgroup generated by the horizontal
hyperbolic translation $\phi_\mu^4$ and the vertical translation
$T(2h)$.  In the quotient by $\phi_\mu^4$ and $T(2h)$, this surface is
topologically a sphere minus four points. Two of the ends of $\Sigma$
are vertical and two of them are horizontal. This surface is similar
to the doubly periodic Scherk minimal surface of $\R^3$.

\begin{proposition}\label{prop:Scherk2p}
  For any $h>\pi$ and any $\mu\in(0,1)$, there exists a properly
  embedded minimal surface $\Sigma$ in $\H^2\times\R$ which is
  invariant by the vertical translation $T(2h)$ and the horizontal
  hyperbolic translation $\phi_\mu^4$ along $\g_0$. In the quotient by
  $T(2h)$ and $\phi_\mu^4$, $\Sigma$ is topologically a sphere minus
  four points, and it has two ends asymptotic to the quotients of
  $\{x>0,t=0\}$ and $\{x>0,t=h\}$, and two ends asymptotic to the
  quotients of $(g(-\mu)\cap\{x<0\})\times[0,h]$ and
  $(g(\mu)\cap\{x<0\})\times[0,h]$. Moreover, $\Sigma$ contains the
  horizontal geodesics $g(\pm\mu)\times\{0\},g(\pm\mu)\times\{h\}$,
  and is invariant by reflection symmetry with respect to $\{t=h/2\}$
  and $\g_{\pi/2}\times\R$.  We call these examples {\em doubly
    periodic Scherk minimal surfaces}.  Finally, we remark that
  $\Sigma$ admits a non-orientable quotient by $\phi_\mu^4$ and
  $T(h)\circ\phi_\mu^2$.
\end{proposition}

\begin{remark}
  When $h<\pi$ and $\mu$ is large enough, we can prove by using the
  maximum principle with vertical catenoids and a fundamental piece of
  the surface $\Sigma$ described in Proposition~\ref{prop:Scherk2p},
  that the corresponding doubly periodic Scherk minimal surface does
  not exist.

  On the other hand, when $h<\pi$ and $\mu$ is small enough, we can
  solve the Plateau problem above in the exterior of certain surface
  ${\cal M}(R,\widetilde\mu)$ described in
  Proposition~\ref{prop:horTHL2}, to prove that the corresponding
  doubly periodic Scherk minimal surface $\Sigma$ exists.
\end{remark}


\subsection{Doubly periodic minimal Klein bottle examples: horizontal
  and vertical Toroidal Halfplane Layers}

In this subsection, we construct non-trivial families of examples of
doubly periodic minimal surfaces.

\medskip

Let us consider the surface $\Sigma_{h,\l}$ constructed in
Subsection~\ref{subsec:M} for $n=2$. By successive extensions by
symmetry along its boundary we get a properly embedded minimal surface
$\Sigma$ which is invariant by the vertical translation $T(2h)$ and
the horizontal translation $\chi_\l^2$, where $(\chi_l)_{l\in(-1,1)}$
is the flow of $Y_{\pi/2}$. The quotient surface by the subgroup of
isometries of $\H^2\times\R$ generated by $T(2h)$ and $\chi_\l^2$ is
topologically a Klein bottle minus two points. The ends of the surface
are asymptotic to vertical annuli. If we consider the quotient by the
group generated by $T(2h)$ and $\chi_\l^4$, we get topologically a
torus minus four points. This example corresponds to the Toroidal
Halfplane Layer of $\R^3$ denoted by $M_{\theta,0,\pi/2}$
in~\cite{Rod}.

\begin{proposition}
  For any $h>0$ and any $\l\in(0,1)$, there exists a properly embedded
  minimal surface in $\H^2\times\R$ invariant by the vertical
  translation $T(2h)$ and the horizontal hyperbolic translation
  $\chi_\l^2$ along $\g_{\pi/2}$, which is topologically a Klein
  bottle minus two points in the quotient by $T(2h)$ and
  $\chi_\l^2$. The surface is invariant by reflection symmetry with
  respect to $\{t=h/2\}$, contains the geodesics $\g_0\times\{0,h\}$,
  $\g_{\pi/2}\times\{0,h\}$ and $\{p_\lambda\}\times\R$, and its ends
  are asymptotic to the quotient of $\g_0\times\R$. Moreover, the
  surface is topologically a torus minus four points when considered
  in the quotient by $T(2h)$ and $\chi_\l^4$.  We call these examples
  {\em horizontal Toroidal Halfplane Layers of type 1}.
\end{proposition}

Let us see another example. This one is similar to the preceding one,
but its ends are now asymptotic to horizontal slices.  We use the
notation introduced in Subsection~\ref{subsec:horizontal}.  For $R>0$,
let $w_R$ be the solution to \eqref{mse} over $\Ome(R)$ with boundary
values zero on $\partial\Ome(R)\setminus \g_0$ and $h/2$ on
$\gamma_0\cap\partial\Ome(R)$.  By the maximum principle,
$w_R<w_{R'}<v$ on $\Omega_R$, for any $R'>R$, where $v$ is the
Abresch-Sa Earp barrier described in Appendix~\ref{app:barrier}.  The
graphs $w_R$ converge as $R\rightarrow +\infty$ to the unique solution
$w$ of \eqref{mse} on $\Ome(\infty)$ with boundary values $h/2$ on
$\overline{q_\mu q_{-\mu}}$ minus its endpoints and $0$ on the
remaining boundary, including the asymptotic boundary at
infinity. (By~\cite{MaRoRo}, we directly know that such a graph exists
and is unique.)

\begin{figure}
  \begin{center}
    \includegraphics[width=0.5\textwidth]{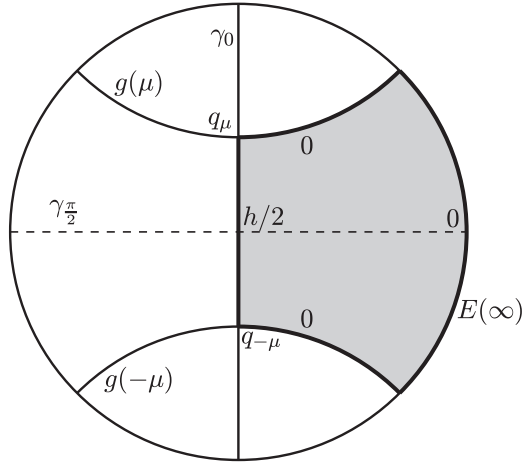}
  \end{center}
  \caption{The minimal surface $\widetilde\Sigma$ of
    Proposition~\ref{prop:vertTHL} is obtained from the vertical
    minimal graph $w$ over $\Ome(\infty)$ (the shadowed domain) with
    the prescribed boundary data.}
  \label{fig:omega_inf_2}
\end{figure}

By uniqueness, we know that such a graph is invariant by reflection
symmetry with respect to the vertical geodesic plane
$\g_{\pi/2}\times\R$. Moreover, the boundary of this graph is composed
of two halves of $g(\mu)$ and $g(-\mu)$ and
$\overline{(q_\mu,0)(q_\mu,h/2)} \cup
\overline{(q_\mu,h/2)(q_{-\mu},h/2)} \cup \overline{
  (q_{-\mu},h/2)(q_{-\mu},0)}$.

If we extend the graph of $w$ by successive symmetries about the
geodesic arcs in its boundary, we obtain a properly embedded minimal
surface $\widetilde\Sigma$ which is invariant by the $\Z^2$ subgroup
$G_1$ of isometries of $\H^2\times\R$generated by $T(h)$ and
$\phi_\mu^4$.  In the quotient by $G_1$, $\widetilde\Sigma$ is a Klein
bottle with two ends asymptotic to the quotient by $G_1$ of the two
horizontal annuli obtained in the quotient of $\H^2\times\{0\}$.  The
quotient by the subgroup generated by $T(2h)$ and $\phi_\mu^4$ gives a
torus minus four points. This example also corresponds to the Toroidal
Halfplane Layer of $\R^3$ denoted by $M_{\theta,0,\pi/2}$
in~\cite{Rod}.

Finally, we remark that taking limits of $\widetilde\Sigma$ as
$h\to+\infty$, we get the horizontal singly periodic Scherk minimal
surface constructed in Subsection~\ref{subsec:horizontal}.

\begin{proposition}\label{prop:vertTHL}
  For any $h>0$ and any $\mu\in(0,1)$, there exists a properly
  embedded minimal surface $\widetilde\Sigma$ in $\H^2\times\R$ which
  is invariant by the vertical translation $T(h)$ and the horizontal
  hyperbolic translation $\phi_\mu^4$ along $\g_0$. In the quotient by
  $T(h)$ and $\phi_\mu^4$, $\widetilde\Sigma$ is topologically a Klein
  bottle minus two points. The ends of $\widetilde\Sigma$ are
  asymptotic to the quotient of $\H^2\times\{0\}$. The surface is
  invariant by reflection symmetry with respect to
  $\g_{\pi/2}\times\R$, and contains the geodesics
  $\g_0\times\{h/2\}$, $\{q_{\pm\mu}\}\times\R$ and
  $g(\pm\mu)\times\{0\}$.  Moreover, in the quotient by $T(2h)$ and
  $\phi_\mu^4$, the surface is topologically a torus minus four points
  corresponding to the ends of the surface (asymptotic to the quotient
  of the horizontal slices $\{t=0\}$ and $\{t=h\}$).  We call these
  examples {\em vertical Toroidal Halfplane Layers of type 1}.
\end{proposition}

\begin{remark}
  \textbf{``Generalized vertical Toroidal Halfplane Layers of type
    1''.}  Consider the domain $\Omega(\infty)$ with prescribed
  boundary data $h/2$ on $\overline{q_\mu q_{-\mu}}$ minus its
  endpoints, $0$ on $(g(\mu)\cup g(-\mu))\cap\{x>0\}$ and a continuous
  function $f$ on the asymptotic boundary $E(\infty)$ of
  $\Omega(\infty)$ at infinity, $f$ vanishing on the endpoints of
  $E(\infty)$ and satisfying $|f|\leq h/2$. By Theorem~4.9
  in~\cite{MaRoRo}, we know there exists a (unique) solution to this
  Dirichlet problem.  By rotating recursively such a graph surface an
  angle $\pi$ about the vertical and horizontal geodesics in its
  boundary, we get a {\em ``generalized vertical Toroidal Halfplane
    Layers of type 1''}, which is properly embedded and invariant by
  the vertical translation $T(h)$ and the horizontal hyperbolic
  translation $\phi_\mu^4$ along $\g_0$.  In the quotient by $T(h)$
  and $\phi_\mu^4$, such a surface is topologically a Klein bottle
  minus two points corresponding to the ends of the surface, that are
  asymptotic to the quotient of a entire minimal graph invariant by
  $\phi_\mu^4$ which contains the geodesics $g(\mu)\times\{0\}$ and
  $g(-\mu)\times\{0\}$. In the quotient by $T(2h)$ and $\phi_\mu^4$,
  the surface is topologically a torus minus four points.
\end{remark}


\subsection{Other vertical Toroidal Halfplane
  Layers}
\label{subsec:horiztorus}

The construction given in this subsection is very similar to the one
considered in Subsection~\ref{subsec:scherk2p}, and we use the
notation introduced there. We consider $h>\pi$ and $\Gamma_R$ the
following Jordan curve:
\[
\begin{array}{ll}
  \Gamma_R = &\overline{(B_0,0)\,(B_R,0)}\cup
  (E(R)\times\{0\})\cup \overline{(C_R,0)\,(C_0,0)}\\
  & \cup \overline{(C_0,0)\,(C_0,h)}\cup \overline{(C_0,h)\,(C_R,h)}\cup
  (E(R)\times\{h\})\\
  & \cup \overline{(B_R,h)\,(B_0,h)}\cup
  \overline{(B_0,h)\,(B_0,0)} .
\end{array}
\]

\begin{figure}
  \begin{center}
    \includegraphics[width=0.5\textwidth]{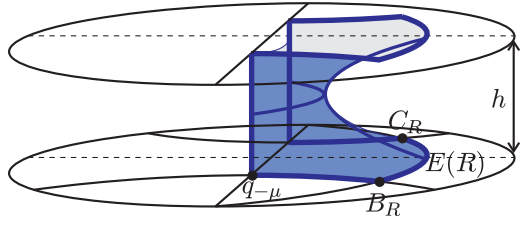}
  \end{center}
  \caption{The embedded minimal disk $\Sigma_R$ bounded by $\G_R$
    (Subsection~\ref{subsec:horiztorus}).}
  \label{fig:Toroidal3}
\end{figure}

$\G_R$ bounds an embedded minimal disk $\Sigma_R$ with minimal
area. As in Subsection~\ref{subsec:scherk2p}, $\Sigma_R$ is a vertical
bigraph with respect to $\{t=h/2\}$. So the sequence of minimal
surfaces $\Sigma_R$, as $R$ varies, satisfies a uniform curvature
estimate far from $\overline{(C_0,0)\,(C_0,h)}$,
$\overline{(B_0,0)\,(B_0,h)}$, $E(R)\times\{0\}$ and
$E(R)\times\{h\}$.

Using the Alexandrov reflection technique with respect to the vertical
planes $g(\nu)\times\R$ as in Subsection~\ref{subsec:horizontal}, we
prove that $\Sigma_R$ is a $Y_0$-bigraph with respect to
$g(0)\times\R=\g_{\pi/2}\times\R$. Thus extending $\Sigma_R$ by
symmetry with respect to $\overline{(B_0,0)\,(B_R,0)}$,
$\overline{(B_0,0)\,(B_0,h)}$ and $\overline{(B_0,h)\, (B_R,h)}$, we
see that a neighborhood of $\overline{(B_0,0)\,(B_0,h)}$ is a
$Y_{\pi/2}$-graph. This neighborhood is then stable and we get
curvature estimates there.  Therefore, the minimal surfaces $\Sigma_R$
satisfy a uniform curvature estimate far from $E(R)\times\{0\}$ and
$E(R)\times\{h\}$.

The surface $\Sigma_R$ is included in $\{x\ge 0\}\times[0,h]$. If
$S_h$ is the same surface as in Subsection~\ref{subsec:scherk2p}
(described in Appendix~\ref{app:barrier}) and $(\chi_l)_{l\in(-1,1)}$
is the flow of $Y_{\pi/2}$, for $l$ close to $1$, $\chi_l(S_h)$ does
not meet $\Sigma_R$. Since $\overline{(B_0,0)\,(B_0,h)}$ and
$\overline{(C_0,h)\,(C_0,0)}$ are the only part of $\Gamma_R$ in
$\H^2\times(0,h)$, we can let $l$ decrease until $l_0<0$, where
$\chi_{l_0}(S_h)$ touches $\partial\Sigma_R$ for the first
time. Actually, $l_0$ does not depend on $R$, and there is two first
contact points: $(B_0,h/2)$ and $(C_0,h/2)$. The surface $\Sigma_R$ is
then between $\chi_{l_0}(S_h)$ and $\g_0\times\R$.

As in Subsection~\ref{subsec:scherk2p}, $\Sigma_R\setminus \Gamma_R$
is a $Z_{\pi/2}$-graph over the open rectangle
$\overline{B_0\,C_0}\times(0,h)$ in $\g_0\times\R$. Then let $R$ tend
to $+\infty$. The surfaces $\Sigma_R$ converge to a minimal surface
$\Sigma_\infty$ satisfying:
\begin{itemize}
\item $\Sigma_\infty$ lies in the region of $\{0\le t\le h\}$ bounded
  by $g(-\mu)\times\R$, $g(\mu)\times\R$, $\g_0$ and
  $\chi_{l_0}(S_h)$.
\item $\Sigma_\infty$ is bounded by four half geodesic lines:
  $\overline{(B_0,0)\,(B_\infty,0)}$,
  $\overline{(B_0,h)\,(B_\infty,h)}$,
  $\overline{(C_0,0)\,(C_\infty,0)}$,
  $\overline{(C_0,h)\,(C_\infty,h)}$, and by two vertical segments:
  $\overline{(B_0,0)\,(B_0,h)}$ and $\overline{(C_0,0)\,(C_0,h)}$.
  Here $B_\infty$ and $C_\infty$ are the limits of the $B_R$ and $C_R$
  as $R\to+\infty$, contained in $\partial_\infty\H^2$.
\item $\Sigma_\infty\setminus\partial\Sigma_\infty$ is a vertical
  bigraph with respect to $\{t= h/2\}$ and a $Z_{\pi/2}$-graph over
  $\overline{B_0\,C_0}\times(0,h)$.
\item $\Sigma_\infty$ is asymptotic to $\{t=0\}$ and $\{t=h\}$.
\end{itemize}

\begin{figure}
  \begin{center}
    \includegraphics[width=0.5\textwidth]{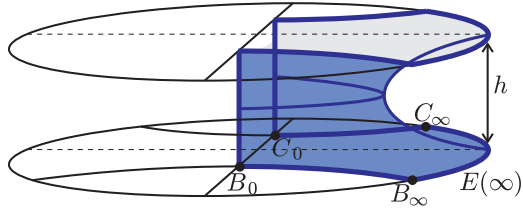}
  \end{center}
  \caption{The embedded minimal disk $\Sigma_\infty$ from which we
    obtain, after successive symmetries with respect to the geodesics
    in its boundary, the doubly periodic example described in
    Proposition~\ref{prop:THLS}.}
  \label{fig:Toroidal4}
\end{figure}

By successive symmetries of $\Sigma_\infty$ with respect to the
geodesics in its boundary, we get an embedded minimal surface $\Sigma$
invariant by the subgroup of isometries of $\H^2\times\R$ generated by
$\phi_\mu^4$ and $T(2h)$. The quotient surface is a torus minus four
points. This example corresponds to a Toroidal Halfplane Layer of
$\R^3$ denoted by $M_{\theta,\pi/2,0}$ in~\cite{Rod}.

  \begin{proposition}\label{prop:THLS}
    For any $h>0$ and any $\mu\in(0,1)$, there exists a properly
    embedded minimal surface $\Sigma$ in $\H^2\times\R$ which is
    invariant by the vertical translation $T(2h)$ and the horizontal
    hyperbolic translation $\phi_\mu^4$ along $\g_0$. In the quotient
    by $T(2h)$ and $\phi_\mu^4$, such a surface is topologically a
    torus minus four points. The ends of $\Sigma$ are asymptotic to
    the quotient of the horizontal slices $\{t=0\}$ and $\{t=h\}$.
    Moreover, $\Sigma$ contains the geodesics $g(\pm\mu)\times\{0\}$,
    $g(\pm\mu)\times\{h\}$ and $\{q_{\pm\mu}\}\times\R$, and is
    invariant by reflection symmetry with respect to $\{t=h/2\}$ and
    $\g_{\pi/2}\times\R$.  Finally, we remark that, in the quotient by
    $\phi_\mu^4$ and $T(h)\circ\phi_\mu^2$, $\Sigma$ is topologically
    a Klein bottle minus two points removed.  We call these examples
    {\em vertical Toroidal Halfplane Layers of type 2}.
  \end{proposition}

  Finally, we observe that, as $h\to+\infty$, $\Sigma$ converges to a
  horizontal singly periodic Scherk minimal surface described in
  Proposition~\ref{prop:Scherk1p}.


\subsection{Other horizontal Toroidal Halfplane Layers}

In this subsection, we also construct surfaces which are similar to
some of Karcher's most symmetric Toroidal Halfplane Layers of
$\R^3$. Now, its ends are asymptotic to vertical planes.

As in the preceding subsection, for $R\ge 0$, we consider the points
$B_R$ and $C_R$ in $g(-\mu)\cap\{x\ge 0\}$ and $g(\mu)\cap \{x\ge 0\}$
at distance $R$ from $\gamma_0$. Let ${\cal P}(R)$ be the polygonal
domain in $\H^2$ with vertices $B_0$, $B_R$, $C_R$ and $C_0$. Let
$u_n$ be the solution to \eqref{mse} defined in ${\cal P}(R)$ with
boundary value $0$ on $\overline{C_R C_0}\cup \overline{C_0 B_0} \cup
\overline{B_0 B_R}$ and $n$ on $\overline{B_R C_R}$. The graph of
$u_n$ is bounded by a polygonal curve. As in
Subsection~\ref{subsec:horizontal}, the sequence converge to a
solution $u_\infty$ of \eqref{mse} on ${\cal P}(R)$ with boundary
value $0$ on $\overline{C_R C_0}\cup \overline{C_0 B_0} \cup
\overline{B_0 B_R}$ and $+\infty$ on $\overline{B_R C_R}$
(by~\cite{NeRo}, we know that it exists and is unique).  The graph of
$u_\infty$, denoted by $\Sigma_R$, is bounded by
$(\{C_R\}\times\R^+)\cup \overline{C_R C_0}\cup \overline{C_0 B_0}
\cup \overline{B_0 B_R} \cup (\{B_R\}\times\R^+)$ and is asymptotic to
$\overline{C_R B_R}\times\R$.

By uniqueness of $u_\infty$, $\Sigma_R$ is symmetric with respect to
$\gamma_{\pi/2}\times\R$. We denote by $\beta_1$ the geodesic
curvature line of symmetry $\Sigma_R\cap (\gamma_{\pi/2}\times\R)$,
and by $F_R$ the intersection point of $\gamma_{\pi/2}$ with
$\overline{B_R C_R}$. We also consider the following points in the
boundary of $\Sigma_R$:
\[
p_1=(\0,0),\quad p_2=(B_0,0),\quad p_3=(B_R,0).
\]
The boundary of $\Sigma_R\cap \{y\le 0\}$ is composed of the union of
the curves $\beta_1$, $\beta_2=\overline{p_1 p_2}$,
$\beta_3=\overline{p_2 p_3}$ and $\beta_4=\{B_R\}\times\R^+$.

The vertical coordinate of the conjugate surface to $\Sigma_R$ is
given by a function $h^*$ defined on ${\cal P}_R$, which is a
primitive of the closed $1$-form $\omega^*$ defined
by~\eqref{def:omega}. We fix the primitive such that $h^*(B_R)=0$ (we
recall that the conjugate surface is well defined up to an isometry of
$\H^2\times\R$.  We can consider $h^*(B_R)=0$ up to a vertical
translation).  By definition of $\omega^*$ and using the fact that
$u_\infty\ge 0$ in ${\cal P}(R)$, we get that $h^*$ increases from $0$
to $h^*(B_0)>0$ along $\overline{B_R B_0}$; it increases from
$h^*(B_0)$ to $h_0=h^*(\0)>h^*(B_0)$ along $\overline{B_0\0}$; $h^*$
is constant along $\overline{\0 F_R}$; and finally $h^*$ increases
from $0$ to $h_0$ along $\overline{B_R F_R}$. In fact, $h_0$ is equal
to the distance from $B_R$ to $F_R$ , i.e. $h_0=h_0(\mu,R)=\frac 1 2
\mbox{dist}_{\H^2}(B_R,C_R)>\ln\frac{1+\mu}{1-\mu}$.

We denote by $\Sigma^*_R$ the conjugate minimal surface of
$\Sigma_R\cap \{y\le 0\}$. We have that
$\partial\Sigma_R^*=\beta_1^*\cup\beta_2^*\cup\beta_3^*\cup\beta_4^*$,
where each $\beta_i^*$ corresponds by conjugation to $\beta_i$.  We
also denote by $p_i^*$ de point in $\partial\Sigma_R^*$ corresponding
by conjugation to $p_i$, $i=1,2,3$.

Up to a vertical translation, we have fixed $p_3^*\in\{t=0\}$.  We can
also take $p_2^*=(\0,h^*(B_0))$, after a horizontal translation.

On the other hand, we know from \cite{HaSaTo} that $\Sigma^*_R$ is a
vertical graph over a domain ${\cal P}(R)^*$, since ${\cal P}(R)$ is
convex.  In particular, $\Sigma_R^*$ is embedded.  We now use the
properties of the conjugation introduced in
Subsection~\ref{subsec:conj} to describe the boundary of $\Sigma_R^*$:
\begin{itemize}
\item $\beta_1^*$ is half a horizontal geodesic with endpoint $p_1^*$.
  Since $p_1^*=(\pi(p_1^*),h_0)$, then we conclude that $\beta_1^*$ is
  contained in $\{t=h_0\}$.

\item The arc $\beta_2^*$ is a vertical geodesic curvature line of
  length $\ln\frac{1+\mu}{1-\mu}$ starting horizontally at $p_2^*$ and
  finishing at $p_1^*$.  In fact, $\beta_2^*$ is the graph of a convex
  increasing function over the (oriented) horizontal geodesic segment
  $\overline{\0\, \pi(p_1^*)}$.  Up to a rotation, we can assume
  $\overline{\0\, \pi(p_1^*)}\subset\g_0^+$.  Since $\beta_1$ and
  $\beta_2$ meet orthogonally at $p_1$ and conjugate surfaces are
  isometric, we get that $\beta_1^*$ is orthogonal to the vertical
  geodesic plane $\g_0\times\R$.  In particular, we can assume up to a
  reflection symmetry with respect to $\g_0\times\R$ that
  $\beta_1^*=g^+(\nu)\times\{h_0\}$, for a certain $\nu\in(0,\mu)$.

\item The curve $\beta_3^*$ is a vertical curvature line of length $R$
  starting horizontally at $p_2^*$ and finishing vertically at
  $p_3^*=(\pi(p_3^*),0)$.  Since $\beta_2,\beta_3$ meet orthogonally
  at $p_2$, the same happens to $\beta_2^*,\beta_3^*$ at $p_2^*$.  In
  particular, $\beta_3^*\subset\g_{\pi/2}\times\R$, and the normal to
  the surface along $\beta_3^*$ is tangent to $\g_{\pi/2}\times\R$.
  Hence $\beta_3^*$ is the graph of a strictly decreasing concave
  function over the (oriented) horizontal segment $\overline{\0\,
    \pi(p_3^*)}\subset\g_{\pi/2}$.  Finally, since
  $\Sigma_R^*\subset\{x>0\}$ in a neighborhood of $\beta_2^*$, we
  deduce $\overline{\0\, \pi(p_3^*)}\subset\g_{\pi/2}^+$.

\item The curve $\beta_4^*\subset\{t=0\}$ is a horizontal curvature
  line with non-vanishing geodesic curvature in $\{t=0\}\equiv\H^2$.
  Since the normal to $\Sigma_R^*$ points to the positive direction of
  the $x$-axis at $p_3^*$ and $\Sigma_R^*\subset\{y>0\}$ in a small
  neighborhood of $\beta_3^*$, we get that $\beta_4^*$ is orthogonal
  to $\g_{\pi/2}\times\R$ and lies inside $\{y>0\}$ near $p_3^*$.
  Moreover, the intrinsic distance in $\Sigma_R\cap\{y\leq 0\}$
  between $\beta_1$ and $\beta_4$ is $h_0$ (which is the asymptotic
  distance at infinity), and $\Sigma_R\cap\{y\leq 0\}$ is isometric to
  $\Sigma_R^*$, then $\beta_4^*$ is asymptotic to $g(\nu)$ at
  $\partial_\infty\H^2$.  This is, $\Sigma_R^*$ is asymptotic to
  $g(\nu)\times [0,h_0]$.  Finally, we know by the maximum principle
  for surfaces with boundary that $\beta_4^*$ is concave with respect
  to ${\cal P}(R)^*$. In particular, it is contained in $\{y>0\}$.
\end{itemize}

\begin{figure}
  \begin{center}
    \includegraphics[width=0.9\textwidth]{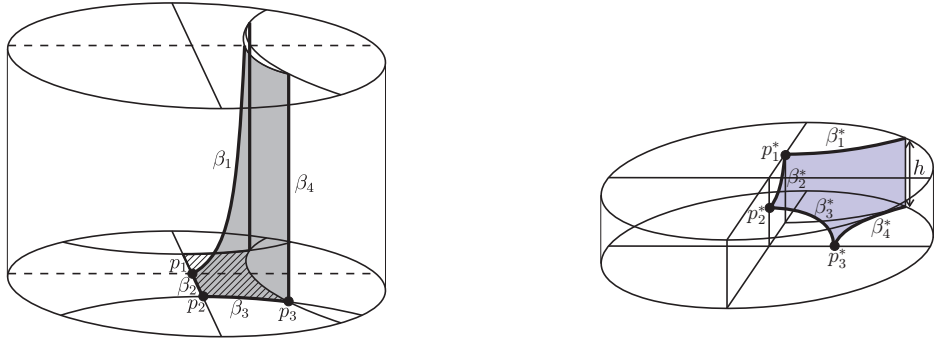}
  \end{center}
  \caption{Left: $\Sigma_R\cap\{y\leq 0\}$. Right: The conjugate
    surface $\Sigma_R^*$ of $\Sigma_R\cap\{y\leq 0\}$, from which we
    obtain after successive symmetries the doubly periodic example
    described in Proposition~\ref{prop:horTHL2}.}
  \label{fig:THL}
\end{figure}

By the maximum principle, $\Sigma_R^*\subset\{0\leq t\leq h_0\}$. If
we make reflection symmetries with respect to $\H^2\times\{0\}$,
$\gamma_0\times\R$ and $\g_{\pi/2}\times\R$, we get a properly
embedded minimal annulus bounded by the geodesics
$g(\pm\nu)\times\{\pm h_0\}$. Then by successive symmetries with
respect to these geodesic boundary lines, we get a doubly periodic
minimal surface invariant by $\phi_\nu^4$ and $T(4h_0)$.  In the
quotient by $\phi_\nu^4$ and $T(4h_0)$, the surface is topologically a
torus minus four points.  In the quotient by $T(4h_0)$ and
$T(2h_0)\circ \phi_\nu^2$, the surface is topologically a Klein bottle
minus two points.  These examples correspond to the Toroidal Halfplane
Layers of $\R^3$ denoted by $M_{\theta,0,0}$ in~\cite{Rod}. We now
have two free parameters instead of only one.

\begin{proposition}\label{prop:horTHL2}
  For any $R>0$ and any $\mu\in(0,1)$, there exist
  $h_0=h_0(R,\mu)>\ln\frac{1+\mu}{1-\mu}$ and
  $\nu=\nu(R,\mu)\in(0,\mu)$ for which there exists a properly
  embedded minimal surface ${\cal M}(R,\mu)$ in $\H^2\times\R$ which
  is invariant by the vertical translation $T(4h_0)$ and the
  horizontal hyperbolic translation $\phi_\nu^4$ along $\g_0$. In the
  quotient by $T(4h_0)$ and $\phi_\nu^4$, ${\cal M}(R,\mu)$ is
  topologically a torus minus four points, whose ends are asymptotic
  to the quotient of $g(\pm\nu)\times\R$. Moreover, ${\cal M}(R,\mu)$
  contains the horizontal geodesics $g(\pm\nu)\times\{\pm h_0\}$, and
  is invariant by reflection symmetry with respect to $\g_0\times\R$,
  $\g_{\pi/2}\times\R$ and $\{t=0\}$. In the quotient by $T(4h_0)$ and
  $T(2h_0)\circ \phi_\nu^2$, ${\cal M}(R,\mu)$ is topologically a
  Klein bottle minus two points.  We call these examples {\em
    horizontal Toroidal Halfplane Layers of type 2}.
\end{proposition}

\begin{remark}
  Up to a hyperbolic horizontal translation along $\g_0$, we can fix
  $B_0=\0$ in the construction above.  Then the graph
  $u_\infty=u_\infty(\mu,R)$ converges as $\mu\to+\infty$ to the
  unique minimal graph $w$ over the geodesic triangle of vertices $\0,
  B_R,q_1=(0,1)$ with boundary values $0$ over
  $\overline{B_R\,\0}\cup\overline{\0,q_1}$ and $+\infty$ over
  $\overline{B_R\,q_1}$.  Such a limit graph produces, after
  successive rotations about the horizontal geodesics
  $\overline{B_R\,\0}\cup\overline{\0,q_1}$ and the vertical geodesic
  $\{B_R\}\times\R^+$ in its boundary, one of the ``horizontal
  helicoids'' ${\cal H}$ described by Pyo in~\cite{Pyo}.  Then the
  conjugate surfaces ${\cal M}(R,\mu)$ converge as $\mu\to+\infty$ to
  one of the ``horizontal catenoid'' constructed in~\cite{MoRo,Pyo}.
\end{remark}


\section{Minimal surfaces invariant by a subgroup of $\Isom(\H^2)$}
\label{sec:isomH2}

In this section, we construct some examples of minimal surfaces
invariant by a subgroup $G$ of the isometries of $\Isom(\H^2\times\R)$
that fix the vertical coordinate. We will say that such a $G$ is a
subgroup of the isometries of $\Isom(\H^2)$. In fact, the subgroups we
consider come from tilings of the hyperbolic plane. We will use some
notation that we introduce in Appendix~\ref{app:tiling}.

The horizontal slices are clearly invariant by any subgroup of the
isometries of $\Isom(\H^2)$. The first non-trivial example is the
following: We consider $n\ge 3$ and $\theta=\pi/n$.  From
Appendix~\ref{app:tiling}, there is $y\in\g_{\theta/2}$ such that the
polygon $\boP_y$ is a regular convex polygon in $\H^2$ with $2n$ edges
of length $2h_n$ and inner angle $\pi/2$ at the vertices (see
Appendix~\ref{app:tiling} for the definitions of $\boP_y$ and
$h_n$). On this polygon, there is a solution $u$ of \eqref{mse} with
boundary values $\pm\infty$ alternatively on each edge. The graph of
$u$ is a minimal surface bounded by $2n$ vertical lines over the
vertices of $\boP_y$. Since $\boP_y$ is the fundamental piece of a
colorable tiling of $\H^2$ (see Proposition~\ref{prop:appendix}) the
graph of $u$ can be extended by successive symmetries along its
boundary to a properly embedded minimal surface in
$\H^2\times\R$. This surface is invariant by the subgroup of
$\Isom(\H^2)$ generated by the symmetries with respect to the vertices
of the tiling.

We now construct other non-trivial examples of properly embedded
minimal surfaces invariant by a subgroup of the isometries of
$\Isom(\H^2)$.  The construction of these surfaces is similar to the
one for some of the most symmetric Karcher's Toroidal Halfplane Layers
in $\R^3$.

Fix $n\ge 3$ and $h>h_n$.  By Claim~\ref{claim:appendix} and
Proposition~\ref{app:tiling}, there exist $\ell<h_n$ and a convex
polygonal domain $\boP(n,h)\subset\H^2$ with $2n$ edges of lengths $h$
and $\ell$, disposed alternately, whose inner angles are $\pi/2$. Such
a domain $\boP(n,h)$ produces by successive rotations about its
vertices a colorable tiling of $\H^2$.

Consider the minimal graph $\Sigma$ over $\boP(n,h)$
with boundary values $0$ over the edges of length $h$
and $+\infty$ over the edges of length $\ell$. Such a graph exists,
by~\cite{NeRo}, and is unique.
    By uniqueness, $\Sigma$ is invariant by reflection symmetry across the
vertical geodesic planes passing
    through the origin of $\boP(n,h)$ and the middle points of the edges of
the polygon.
We rotate $\Sigma$ about the horizontal and vertical geodesics in its
boundary,
producing a properly embedded minimal surface $\cal M$ invariant by a
subgroup of the group of isometries of the tiling produced from $\boP(n,h)$
. $\cal M$ projects
vertically over the whole $\H^2$,
and contains all
the edges of the tiling coming from the edges of $\boP(n,h)$ of length $h$
(identifying them with the corresponding horizontal geodesics at height
zero),
and the vertical geodesics over the vertices of the tiling.

\begin{proposition}\label{prop:newTHL}
  For any $n\ge 2$ and any $h>h_n$, there exists a properly embedded
  minimal surface $\cal M$ invariant by the group of isometries of the
  tiling produced by the polygon $\boP(n,h)$ defined above.  The
  vertical projection of $\cal M$ is the whole $\H^2$ and the ends of
  $\cal M$ are asymptotic to the vertical geodesic planes over the
  edges of the tiling coming from the edges of $\boP(n,h)$ with length
  $\ell$.  Moreover, $\cal M$ contains all the edges of the tiling
  coming from the edges of length $h$ and the vertical geodesics over
  the vertices of the tiling.
\end{proposition}

In the following subsections, we prove:
\begin{proposition}\label{prop:new}
  For any $n\ge 3$ and any $h>h_n$, there exists a properly embedded
  minimal surface $M$ invariant by the group of isometries of the
  tiling produced by the polygon $\boP(n,h)$.  $M$ projects vertically
  over the tiles in black and its ends are asymptotic to the vertical
  geodesic planes over the edges of the tiling coming from the edges
  of $\boP(n,h)$ of length $h$.  Moreover, $M$ is invariant by
  reflection symmetry across $\{t=0\}$ and contains the vertical
  geodesics over the vertices of the tiling.
\end{proposition}


\subsection{The conjugate minimal surfaces $M_{h,\l}^*$}

Let $n\ge 3$ be an integer and $\theta=\pi/n$. We consider $h>0$ and
$\l\in(0,1)$. In Subsection~\ref{subsec:M}, we have constructed the
minimal surface $M_{h,\l}$ which is bounded by the union of four
curves: $c_1$, $c_2$, $c_3$ and $c_4$.

Let $M_{h,\l}^*$ be the conjugate minimal surface of $M_{h,\l}$. The
aim of this subsection is to describe $M_{h,\l}^*$ and prove that it
is embedded.  We notice that $M_{h,\l}^*$ is well defined up to an
isometry of $\H^2\times\R$. In the following, we will fix this
isometry by making some hypotheses on $M_{h,\l}^*$.

The vertical coordinate $h^*$ of $M_{h,\l}^*$ is defined on $\Omega_0$
by a primitive of the closed $1$-form $\omega^*$ defined
in~\eqref{def:omega}.  Up to a vertical translation, we can assume
$h^*(p_\l)=0$. Because of the definition of $\omega^*$ and since
$M_{h,\l}\subset \H^2\times[0,h/2]$, $h^*$ increases along $\pi(c_4)$
from $p_\l$ to $q_1$, along $c_2$ from $p_\l$ to $\0$ and along $c_1$
from $\0$ to $q_1$.  Thus $h^*$ is non-negative.

The surface $M_{h,\l}^*$ is bounded by $c_1^*,c_2^*,c_3^*,c_4^*$,
where each $c_i^*$ corresponds by conjugation to $c_i$. Let us give a
first description of these curves (see Figure~\ref{fig:Mhl*}):
\begin{itemize}
\item $c_1^*$ is a vertical geodesic curvature line lying on a
  vertical geodesic plane~$\Pi_1$, with infinite length and endpoint
  $A_0^*$, the conjugate point to $A_0$. We can assume that $A_0^*$ is
  the point $(\0,h^*(\0))$ and that $\Pi_1$ is the plane
  $\g_0\times\R$. The unit tangent vector to $c_1^*$ at $A_0^*$ is
  horizontal and we assume it points to $\{y\ge 0\}$.  The angle
  function $\nu^*$ is positive along $c_1^*$ (as this was the case for
  the angle function $\nu$ of $M_{h,\l}$ along $c_1$) and the height
  function increases along $c_1^*$ when starting from $A_0^*$. In the
  Euclidean plane $\Pi_1$, $c_1^*$ is then the graph of a convex
  increasing function over a part $[\0,a_1)$ of $\g_0^+$ ($a_1$ could
  be a priori in the asymptotic boundary of $\H^2$).
\item $c_2^*$ is a vertical geodesic curvature line of length
  $\ln\left(\frac{1+\l}{1-\l}\right)$ lying on a vertical geodesic
  plane~$\Pi_2$. Since, the angle between $c_1$ and $c_2$ is $\t$ at
  $A_0$, we get that the angle between $\Pi_1$ and $\Pi_2$ is $\theta$
  ($M_{h,\l}^*$ is horizontal at $A_0^*$ and isometric to
  $M_{h,\l}$). We take $\Pi_2$ the vertical plane
  $\pi^{-1}(\g_\t)$. Now $M_{h,\l}^*$ is uniquely defined.  Starting
  from $A_0^*$, the height function decreases along $c_2^*$ from
  $h^*(\0)$ to $h^*(p_\l)=0$. In the Euclidean plane $\Pi_2$, $c_2^*$
  is then the graph of a concave decreasing function over a part of
  the geodesic $\gamma_\t^+$. We denote by $A_2^*$ the endpoint of
  $c_2^*$ which is different from $A_0^*$. We have $A_2^*=(a_2,0)$,
  with $a_2\in\g_\t^+$.
\item $c_3^*$ is a horizontal geodesic curvature line of length $h/2$
  at height zero, going from $A_2^*$ to a point $A_3^*=(a_3,0)$.  The
  unit tangent vector to $c_3^*$ at $A_2^*$ is orthogonal to $\Pi_2$
  and points into the side of $\Pi_2$ that contains $c_1^*$. As a
  curve of $\H^2\times\{0\}$, the geodesic curvature of $c_3^*$ never
  vanishes. In fact, since the normal vector field of $M_{h,\l}$
  rotates less than $\pi$ along $c_3$, the total geodesic curvature of
  $\pi(c_3^*)\subset \H^2$ is less than $\pi$. This implies that
  $\pi(c_3^*)$ and $c_3^*$ are embedded and $\pi(c_3^*)$ does not
  intersect $\overline{\0a_2}$.
\item $c_4^*$ is the half vertical geodesic line $\{a_3\}\times\R^+$.
\end{itemize}

\begin{figure}
  \begin{center}
    \resizebox{0.9\linewidth}{!}{\input{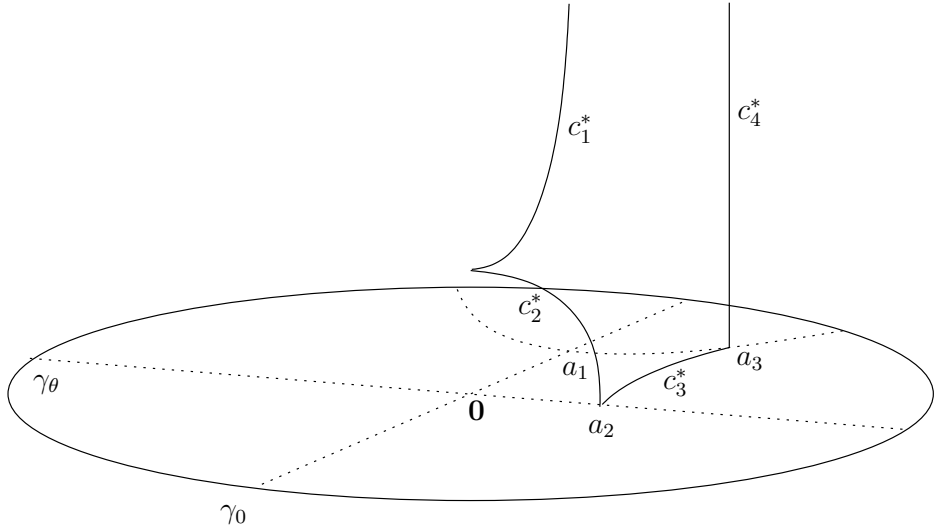}}
    \caption{The boundary of $M_{h,\l}^*$\label{fig:Mhl*}}
  \end{center}
\end{figure}

We know that the distance between $c_1$ and $c_4$ is uniformly bounded
and the surface is isometric to its conjugate, so the same is true for
$c_1^*$ and $c_4^*$. Thus the distance between $a_1$ and $a_3$ is
bounded. This is, $a_1$ is in $\H^2$, not in $\partial_\infty\H^2$.
Then, seeing in the Euclidean plane $\Pi_1$, $c_1^*$ is the graph of a
convex increasing function over a part $[\0,a_1)$ of $\g_0^+$ with
limit $+\infty$ at $a_1$.

Because of the asymptotic behaviour of $M_{h,\l}$ near $q_1$,
$M_{h,\l}^*$ is asymptotic to $\overline{a_1a_3}\times\R$, and the
geodesic $\overline{a_1a_3}$ has length $h/2$. Besides, since the
normal vector to $M_{h,\l}^*$ lies in $\Pi_1$ along $c_1^*$, the
geodesic $\overline{a_1a_3}$ is orthogonal to $\gamma_0$ at $a_1$, and
$a_3$ lies in $\{x\ge 0\}$.

Let $(\phi_l)_{l\in(-1,1)}$ be the flow given by $Y_0$. Let $\gamma$
be the complete geodesic of $\H^2$ that contains $a_1$ and $a_3$.  We
know that $\g$ is orthogonal to $\g_0$.  We consider the foliation of
$\H^2\times\R$ by the vertical geodesic planes
$\phi_l(\gamma\times\R)$. Since every point in $M_{h,\l}^*$ is at a
bounded distance from its boundary, for $l$ close to $1$ we have
$\phi_l(\gamma\times\R)\cap M_{h,\l}^*=\emptyset$. Let $l$ decrease
until a first contact point for $l=l_0$. Since $M_{h,\l}^*$ is
asymptotic to $\overline{a_1a_3}\times\R$, either $l_0=0$ or $l_0>0$.
Let us assume $l_0>0$ and reach a contradiction. We have two cases:
the first contact point is contained on $c_3^*$ or it coincides with
$A_2^*$. In the first case we get a contradiction using the maximum
principle, since the normal vector field of the surface is horizontal
along $c_3^*$ and $\phi_{l_0}(\gamma\times\R)$ is on one side of
$M_{h,\l}^*$.  Let us now assume that the first contact point is
$A_2^*$. The unit tangent vector to $c_3^*$ points into $\cup_{l\geq
  l_0}\phi_l(\gamma\times\R)$ at $A_2^*$, this contradicts that we
have the first contact point for $l=l_0$. So $M_{h,\l}^*$ never
intersects $\phi_l(\gamma\times\R)$ until $l=0$. This implies that
$a_2$ and $c_3^*$ are in the half hyperbolic space bounded by $\gamma$
which contains $\0$.

Let $\gamma'$ be the geodesic passing through $a_3$ and orthogonal to
$\gamma_\theta$. Using a similar argument as above with the
corresponding foliation by vertical geodesic planes, we can prove
that:
\begin{itemize}
\item $c_3^*$ is in the half hyperbolic plane $\{x\ge 0\}$;
\item $c_3^*$ and $a_2$ are in the half hyperbolic plane bounded by
  $\gamma'$ which contains~$\0$.
\end{itemize}
For the second item, we need to extend $M_{h,\l}^*$ by symmetry along
$c_2^*$.

Let $\Ome$ be the domain of $\H^2$ bounded by $\overline{a_3 a_1}$,
$\overline{a_1 \0}$, $\overline{\0 a_2}$ and $c_3^*$. Since the angle
function $\nu^*$ never vanishes outside $c_3^*\cup c_4^*$, we conclude
$M_{h,\l}^*\subset \Ome\times\R$. In fact, since $A_0^*$ is the only
point in $M_{h,\l}^*$ that projects on $\0$, $M_{h,\l}^*$ is a
vertical graph over $\Ome$. This implies that $M_{h,\l}^*$ is
embedded.


\subsection{Symmetry and period problem}

We recall that $n \ge 3$.  From now on, we assume that $h>h_n$, where
$h_n$ is defined in Appendix~\ref{app:tiling} ($h_n$ is the length of
the edges of the regular geodesic polygon with $2n$ edges with
interior angles $\pi/2$).  We want to find a value for the parameter
$\l$ for which we can construct an embedded minimal surface extending
$M_{h,\l}^*$ by symmetry along its boundary.

Let us consider the surface $\Sigma_{h,\l}$ described in
Subsection~\ref{subsec:M}. The same argument as in this subsection
proves that $\Sigma_{h,\l}$ converges when $\l\rightarrow 1$. By
uniqueness, we get that this limit minimal surface must be $\Sigma_h$,
described in Subsection~\ref{subsec:vertical} (see
  Figure~\ref{fig:plateau2}). Moreover, the surfaces $\Sigma_{h,\l}$
depend continuously of the parameter $\l$. Thus $a_1$, $a_2$ and $a_3$
depend continuously of $\l$ as well.

We define $M_h=\Sigma_h\cap \{0\le t\le h/2\}$, and $M_h^*$ its
conjugate surface.  As both $M_{h,\l}$ and $M_{h,\l}^*$ are vertical
minimal graphs and $M_{h,\l}$ converges to $M_h$ as $\l\to 1$, we can
conclude as in~\cite{MoRo} that the graphs $M_{h,\l}^*$ converge to
$M_h^*$ when $\l\to 1$.

We translate vertically $M_h^*$ so that $A_0^*=(\0,0)$. The curve
$M_h\cap \{t=h/2\}$ corresponds by conjugation to a vertical geodesic
$\{a'\}\times\R$, where $a'$ is the limit of the points $a_3$ when
$\l\to 1$. Since $M_h$ is invariant by the reflection symmetric with
respect to the plane $\gamma_{\t/2}\times\R$, then $M_h^*$ is
invariant by the rotation of angle $\pi$ about the geodesic
$\g_{\t/2}$, contained in $M_h^*$. Therefore $a'\in \g_{\t/2}$ and
this implies that, for $\l$ sufficiently close to $1$, $a_3$ lies in
the hyperbolic angular sector $T_\t=\{(r\sin u,r\cos u)\in\H^2,
r\in[0,1),u\in[0,\t]\}$.

Let $a_4$ be the orthogonal projection of $a_3$ over $\gamma_\t$. As
$\l$ goes to $1$, $a_3$ goes to $a'$ and $a_4$ goes to the projection
$a'_\t$ of $a'$. We recall that $a_1$ is the orthogonal projection of
$a_3$ on $\gamma_0$ so $a_1$ goes to the projection $a'_0$ of $a'$ on
$\g_0$. Since $h>h_n$ and $M_h^*$ (for $\lambda=1$) is invariant by
the rotation of angle $\pi$ about $\g_{\t/2}$, we deduce that the
angle between $\overline{a'a'_0}$ and $\overline{a'a'_\t}$ is strictly
smaller than $\pi/2$. Thus the angle between $\overline{a_3a_1}$ and
$\overline{a_3a_4}$ is strictly less than $\pi/2$, for $\l$ close
to~1.

\medskip

Let us observe what happens when $\l$ is close to $0$. By
construction, $a_3$ is at distance $h/2$ from the geodesic $\g_0$
(i.e. $a_3$ lies on $d(h/2)$, the equidistant curve of $\g_0$ at
distance $h/2$). Besides the distance from $\0$ to $a_3$ is less than
the sum of the lengths of $c_2^*$ and $c_3^*$. So this distance is
less than $\ln\left(\frac{1+\l}{1-\l}\right)+h/2$. So for $\l$ small,
the distance between $\0$ and $a_3$ is close to $h/2$. This implies
that $a_3$ lies outside the angular sector $T_\t$ when $\l$ is close
to zero.

By continuity, there is a largest $\l$, denoted by $\l_0$, such that
$a_3\in \gamma_\theta$. In particular, $a_3$ is contained in $T_\t$
for any $\l>\l_0$.  For $\l>\l_0$ close to $\l_0$, $a_3\in T_\t$ is
close to $\g_\t$. So the angle between $\overline{a_3a_1}$ and
$\overline{a_3a_4}$ is bigger than $\pi/2$. A continuation argument
says that there exists $\l_1\in(\l_0,1)$ such that $a_3\in T_\t$ and
the angle between $\overline{a_3a_1}$ and $\overline{a_3a_4}$ is equal
to $\pi/2$ (see the proof of Claim~\ref{claim:appendix} for a similar
argument). This value $\l_1$ is the one we look for; so from now on,
we fix $\l=\l_1$.

The domain $\Ome$ is included in the convex polygonal domain of
vertices $\0$, $a_1$, $a_3$ and $a_4$.  We denote by
$\widetilde{\Ome}$ the domain obtained from $\Ome$ by reflection with
respect to the geodesics $\g_0$ and $\g_\t$ successively. The boundary
of $\widetilde{\Ome}$ has $2n$ vertices which are the images of $a_3$
and is composed of $n$ geodesic arcs corresponding to
$\overline{a_1a_3}$ and $n$ concave arcs corresponding to
$c_3^*$. This domain is included in the convex polygonal domain
$\boP$, which is constructed by the same symmetries from the geodesic
polygon of vertices $\0, a_1, a_3, a_4$ (this polygon corresponds to
the polygon $\boP_{a_3}$ in Appendix~\ref{app:tiling}). $\boP$ has
$2n$ vertices coming from $a_3$, all of them with interior angle
$\pi/2$; and its edges have lengths $h$ and $b$, alternatively, where
$b$ is twice the length of the geodesic arc $\overline{a_3a_4}$. Such
a polygon $\boP$ is then the fundamental piece of a colorable tiling
of $\H^2$ (see Proposition~\ref{prop:appendix}).

Let us now extend $M_{h,\l_1}^*$ by successive reflection symmetries
with respect to the planes $\g_0\times\R$ and $\g_\t\times\R$. We get
a minimal surface $\widetilde{M}$ which is a vertical graph over
$\widetilde{\Ome}$ with value $0$ along the concave arcs and $+\infty$
on the geodesic arcs.  Moreover, this surface is in $\{t\ge 0\}$ and
has all the symmetries of the polygonal domain $\boP$. By reflection
symmetry with respect to the horizontal slice $\{t=0\}$, we get an
embedded minimal surface whose boundary consists of $2n$ vertical
geodesic lines passing through the vertices of $\boP$. Such a surface
is topologically a sphere minus $n$ points.

From Proposition~\ref{prop:appendix}, $\boP$ is the fundamental piece
of a hyperbolic colorable tiling. Thus we can extend the surface by
successive reflection symmetries along the vertical geodesics
contained in its boundary, getting a properly embedded minimal surface
$M$ which is invariant by the group of symmetries generated by the
rotation around the vertices of the tiling. Moreover the surface
projects only on tiles in black of $\boP$. This proves
Proposition~\ref{prop:new}.

\begin{remark}
  If $n=2$, the above contruction can be done without selecting the
  value of the parameter $\lambda$. Thus we get the surface
  $\widetilde{M}$ that can be extended by symmetry with respect to
  $\{t=0\}$ to get a minimal surface whose boundary consists of $4$
  vertical geodesic lines. This surface is topologically an
  annulus. So this surface is a solution to the following Plateau
  problem: finding a minimal annulus bounded by four vertical geodesic
  vertical lines. In this sense, it is very similar to the Karcher
  saddle \cite{Kar} of $\R^3$. But in our situation it can't be
  extended by symmetry along its boundary into an embedded minimal
  surface of $\H^2\times\R$.
\end{remark}



\appendix

\section{Geodesic polygonal domains with right angles}\label{app:tiling}

In this appendix, we give some facts about the tilings of the
hyperbolic plane that we consider in the paper.

Let $n\geq 3$ be an integer and define $\t=\pi/n$. Let $y_l$ be the
point $(l\sin(\t/2),l \cos(\t/2))$ in $\H^2$, for $0<l<1$. Rotating
$y_l$ around $\0$ by $k\t$ ($k=1,\cdots,2n-1$), we get the $2n$
vertices of a regular convex geodesic polygon in $\H^2$. We denote by
$h$ the length of one of its $2n$ edges. $h$ is an increasing function
of $l$. When $l$ varies from $0$ to $1$, the interior angle of the
polygon at $y_l$ decreases from $\pi-\theta$ to $0$. Thus there is one
value of $l$ such that this angle is $\pi/2$. We denote by $h_n$ the
associated value of $h$.

Let $y$ be in $T_\t$. Considering the successive image of $y$ by the
reflections with respect to $\gamma_{k\t}$ (($k=1,\cdots,2n$), we
construct the $2n$ vertices of a convex polygon whose edges has
alternative lengths $a_y$ and $b_y$, where $a_y/2$ is the distance
from $y$ to $\gamma_0$ and $b_y/2$ the one to $\g_\t$. We denote by
$\boP_y$ this polygon and by $\alpha_y$ the interior angle of $\boP_y$
at the vertex $y$ (the angle is the same at every vertex).

\begin{claim}\label{claim:appendix}
  For any $a\ge h_n$, there is $y\in T_\t$ such that $a_y=a$ and
  $\alpha_y=\pi/2$.
\end{claim}
\begin{proof}
  Let $d(a/2)$ be the equidistant curve to $\gamma_0$ at distance
  $a/2$ in $\{x\ge 0\}$. Let $y$ be on the part of $d(a/2)$ between
  $\gamma_{\t/2}$ and $\gamma_\t$. Then $a_y=a$. If
  $y\in\gamma_{\t/2}$, $\boP_y$ is a regular convex polygon
  ($a_y=b_y$) and $\alpha_y\le \pi/2$, since $a\ge h_n$. For $y$ close
  to $\g_\t$, $\alpha_y>\pi/2$. By continuity, there is $y$ such that
  $\alpha_y=\pi/2$.
\end{proof}

\begin{proposition}\label{prop:appendix}
  Let $y\in T_\t$ such that $\alpha_y=\pi/2$. Then $\boP_y$ is the
  fundamental piece of a tiling of $\H^2$. This tiling is given by
  considering the successive images of $\boP_y$ by reflection with
  respect to its edges. Moreover, this tiling is colorable i.e. we can
  associate to any tile a color (black or white) such that two tiles
  having a common edges do not have the same color.
\end{proposition}

For such a tiling, every vertex lies in four tiles: two are black and
two are white. Two tiles of the same color with a common vertex are
exchanged by the symmetry around this vertex.
Proposition~\ref{prop:appendix} is a consequence of Poincar\'e's
polyhedron Theorem \cite{Mas}.


\section{Some interesting minimal surfaces}\label{app:barrier}

In this appendix, we recall some known minimal surfaces in
$\H^2\times\R$ that we used in the paper.

Let us consider the half-space model for the hyperbolic plane :
$\H^2=\{(x_1,x_2)\in\R\times\R^+\}$ with the hyperbolic metric
$g=\frac{1}{x_2^2}(dx_1^2+dx_2^2)$.

On $\{x_1>0\}$, the function $v(x_1,x_2)=\log(
\frac{\sqrt{x_1^2+x_2^2}+x_2}{x_1})$ is a solution to \eqref{mse}. Its
graph is then a minimal surface in $\H^2\times\R$. On the boundary of
$\{x_1>0\}$, $v$ takes the value $+\infty$ on the geodesic line
$\{x_1=0\}$ and takes value $0$ on the asymptotic boundary of
$\{x_1>0\}$. This solution was discovered independently by U.~Abresch
and R.~Sa~Earp. This surface is used in
Subsection~\ref{subsec:horizontal}

On the whole $\H^2$, another solution to \eqref{mse} is given by
$u_a(x_1,x_2)=a\log (x_1^2+x_2^2)$. This solution is invariant by the
$Z$-flow, for $Z$ normal to $\{x_1=0\}$. In fact the graph of $u_a$ is
a minimal surface foliated by horizontal geodesics in $\H^2\times\R$
normal to $\{x_1=0\}\times\R$. Adding a constant $c$ to $u_a$, we
create a foliation of $\H^2\times\R$ by such surfaces. When $a$ varies
in $\R$, we get a family of minimal surfaces which are similar to
planes in $\R^3$.  Moreover, for any non vertical tangent plane at
$(0,1,0)$ which is tangent to $Z$, one surface in this family is
tangent to this tangent plane. In order to have the complete family,
we can add the vertical minimal plane
$\{x_1^2+x_2^2=1\}\times\R$. These surfaces are the $P$ surfaces used
in the proof of Claim~\ref{claim:Zgraph}.

If we look for solutions of \eqref{mse} of the form
$u(x_1,x_2)=f(x_1/x_2)$, we obtain solutions which are invariant by a
translation along the geodesic $\{x_1=0\}$. The above solution $v$ is
one such solution. In fact, for any $h>\pi$, there is $d_h>0$ and a
function $f_h$ which is defined on $[d_h,+\infty)$ such that
$u_h(x_1,x_2)=f_h(x_1/x_2)$ is a solution to \eqref{mse} (see
\cite{SaE,MaRoRo}). This function $f_h$ is a decreasing function with
$f_h(d_h)=h/2$ and $\lim_{+\infty} f_h=0$ and $\lim_{d_h}
f_h'=-\infty$.  The function $u_h$ is then defined on the set of
points at distance larger than $d_h$ from $\{x_1=0\}$ and has boundary
value $h/2$ on the equidistant and $0$ on the asymptotic
boundary. When $h\rightarrow +\infty$, $u_h$ converge to the above
solution $v$. The graph of $u_h$ is a minimal surface inside $\{0<t\le
h/2\}$ which is foliated by horizontal equidistant lines to
$\{x_1=0\}\times\R$ and is vertical along its boundary. Then this
graph can be extended by symmetry with respect to $\{t=h/2\}$ to a
complete minimal surface $S_h$ which is a vertical bigraph, included
in $\{0<t<h\}$, foliated by horizontal equidistant lines to
$\{x_1=0\}\times\R$. Moreover, the supremum of the vertical gap on
$S_h$ is $h$. The surfaces $S_h$ are used in
Subsections~\ref{subsec:scherk2p} and \ref{subsec:horiztorus} as
barriers in our construction.


\section{Alexandrov reflection}
\label{app:alexandrov}

In Subsection~\ref{subsec:scherk2p}, we construct a minimal surface
$\Sigma_\infty$ as the limit of surfaces $\Sigma_R$. These surfaces
$\Sigma_R$ are minimal disks bounded by a Jordan curve $\G_R$. We say
that Alexandrov reflection technique can be applied with respect to
horizontal slices to prove that $\Sigma_R$ is a vertical bigraph with
respect to $\{t=h/2\}$. Since there are two vertical arcs in $\G_R$,
we need to explain how the classical Alexandrov reflection technique
works along these vertical edges.

\medskip

In order to lighten the notation, we put $\Sigma=\Sigma_R$ and
$\G=\G_R$.  For $l\in[0,h]$, we define $\Pi_l$ the horizontal slice
$\{t=l\}$. We denote by $P_l$ and $Q_l$ the points in the vertical
edges of $\G$ at height $l$ (since the arguments work the same for
both points in the sequel, we will assume that there is only one). Let
$\Sigma_l=(\Sigma\cap \pi_l)\setminus \{P_l,Q_l\}$. We also define
$\Sigma_l^+$ (resp.  $\Sigma_l^-$) the part of $\Sigma$ above
(resp. below) $\Pi_l$ minus its boundary. Finally we denote by
$\Sigma_l^{+*}$ and $\Sigma_l^{-*}$ the symmetric of $\Sigma_l^+$ and
$\Sigma_l^-$ by $\Pi_l$.

\medskip

The main step of the Alexandrov reflection technique is to prove that,
for any $l\in (h/2,h]$, $\Sigma_l^-\cap \Sigma_l^{+*}= \emptyset$ and
$\Sigma$ is never vertical along $\Sigma_l$.

The property is true for $l=h$ since $\Sigma_h^{+*}=\emptyset$ and
$\Sigma$ is inside the convex hull of its boundary.

We notice that for any $l\in (h/2,h)$, if $\Sigma_l^-\cap
\Sigma_l^{+*}= \emptyset$ is proved, then $\Sigma$ is never vertical
along $\Sigma_l$ follows easily.

Now we consider $l_0\in (h/2,h]$ such that the property is satisfied
for any $l\ge l_0$. Let us assume that there exists a sequence of
$l_k<l_0$ with $l_k \rightarrow l_0$ and, for any $k$, there is $p_k\in
\Sigma_{l_k}^-\cap \Sigma_{l_k}^{+*}$.

Since $\Sigma_{l_0}^-\cap \Sigma_{l_0}^{+*}=\emptyset$, the limit
$p_\infty$ of $p_k$ is either in $\Sigma_{l_0}$ or in the vertical
edge. Since $\Sigma$ is not vertical along $\Sigma_{l_0}$,
$p_\infty\notin \Sigma_{l_0}$. So $p_\infty$ is in the vertical
edge. Since $\Sigma_{l_0}^-\cap \Sigma_{l_0}^{+*}=\emptyset$, the
tangent space to $\Sigma_{l_0}^-$ and $\Sigma_{l_0}^{+*}$ are
different for any point in the the vertical edge except at $P_{l_0}$
so the only possible limit is $p_\infty=P_{l_0}$.

Let us first consider the case $l_0<h$, and let $(x,y,z)$ be an
orthogonal coordinate system at $P_{l_0}$ such that $(x,y)$ are
euclidean coordinates in the vertical plane tangent to $\Sigma$ at
$P_{l_0}$, where $\partial_x$ is a vertical down pointing vector field
and $\partial_y$ is a horizontal vector field.  $\Sigma$ is then
locally the graph of a function $z=w(x,y)$ over $\{y\ge 0\}$. $w$
vanishes on $\{y=0\}$ and has vanishing differential at the origin. We
notice that $\{z=0\}$ is a minimal surface thus from the proof of
Theorem 5.3 in \cite{CoMi2}, $w$ can be written in the following way:
$$
w(x,y)=p(x,y)+q(x,y) ,
$$
where $p$ is a homogeneous harmonic polynomial of degree $d$ and $q$
satisfies
$$
|q(X)|+|X||\nabla q(X)|+\cdots+|X|^d|\nabla^dq(X)|\le C|X|^{d+1} .
$$

Since $\Sigma_{l_0}^-\cap \Sigma_{l_0}^{+*}=\emptyset$, then
$w(x,y)-w(-x,y)$ has a sign for any $|(x,y)|<\eps$ with $x\neq 0$ and
$y\neq 0$. We assume that the coordinate $z$ is chosen such that this
sign is $+$. Thus $0\le w(x,y)-w(-x,y)=p(x,y)-p(-x,y)+q(x,y)-q(-x,y)$
for $x$ and $y$ non negative close to $0$, and it does not vanish for
positive values of $x$ and $y$. Thus the degree of $p(x,y)-p(-x,y)$
has to be $2$, and $p(x,y)=\alpha xy$ with $\alpha>0$.

When $l_0=h$, we also get that $\Sigma$ is the graph of function $w$
over $[0,\eps]^2$ with $w(x,y)=\alpha xy+q(x,y)$, for the same choice
of coordinate system, with $\alpha$ and $q$ satifying the same
hypotheses as above.

Since $p_k\rightarrow P_{l_0}$, for $k$ large enough we get
$p_k=(x_k,y_k,w(x_k,y_k))$, with $(x_k,y_k) \in
[-l_k,\eps]\times[0,\epsilon]$. Since $p_k\in \Sigma_{l_k}^-\cap
\Sigma_{l_k}^{+*}$, we have :
\begin{equation}\label{egalite}
  w(x_k,y_k)=w(2(l_0-l_k)-x_k,y_k)
\end{equation}
But if $(x,y)\in[\lambda,\eps]\times[0,\eps]$ we have:
\begin{align*}
w(x,y)-w(2\lambda-x,y)&\ge
2\alpha(x-\lambda)y-2\sup_{u\in[-\eps,\eps]}
|\partial_xq(u,y)|(x-\lambda)
\end{align*}
Since $0=w(x,0)=q(x,0)$, we get
$$
|\partial_xq(u,y)|\le \sup_{v\in[0,\eps]}|\partial_y\partial_x
q(u,v)|y\le C\sqrt{u^2+\eps^2}y
$$
Thus
\begin{align*}
w(x,y)-w(2\lambda-x,y)&\ge
2\alpha(x-\lambda)y-C2\sqrt{2}\eps(x-\lambda)y\\
&\ge 2[\alpha-\sqrt{2}C\eps](x-\lambda)y ,
\end{align*}
which is positive if $x>\lambda$, $y>0$ and $\eps$ is small
enough. This contradicts \eqref{egalite} when $k$ is large enough.

We then have proved that: for any $l\in (h/2,h]$, $\Sigma_l^-\cap
\Sigma_l^{+*}= \emptyset$ and $\Sigma$ is never vertical along $\Sigma_l$.

\medskip

Therefore, we obtain that either $\Sigma_{h/2}^- =\Sigma_{h/2}^{+*}$
and it is a vertical bigraph with respect to $\{t=h/2\}$, or
$\Sigma_{h/2}^-$ and $\Sigma_{h/2}^{+*}$ are two non intersecting
minimal surfaces with the same boundary. In this second case,
$\Sigma_{h/2}^-$ is clearly below $\Sigma_{h/2}^{+*}$ along the
$\G\cap \Pi_0$. By symmetry by $\Pi_{h/2}$ this implies that
$\Sigma_{h/2}^+$ is below $\Sigma_{h/2}^{-*}$ along $\G\cap
\Pi_h$. But doing Alexandrov reflection technique as above with the
slices $\Pi_l$, $l\in[0,h/2]$, we get that $\Sigma_{h/2}^+$ is above
$\Sigma_{h/2}^{-*}$ along $\G\cap \Pi_h$.  Finally, we have proved
$\Sigma_{h/2}^-=\Sigma_{h/2}^{+*}$.


\textsc{Laurent Mazet}

Universit\'e Paris-Est,

Laboratoire d'Analyse et Math\'ematiques Appliqu\'ees, UMR 8050

UFR de Sciences et technologies, D\'epartement de Math\'ematiques

61 avenue du g\'en\'eral de Gaulle 94010 Cr\'eteil cedex (France)

\texttt{laurent.mazet@univ-paris12.fr}

\medskip

\textsc{M. Magdalena Rodr\'\i guez}

Universidad de Granada,

Departamento de Geometr\'\i a y Topolog\'\i a

Campus de Fuentenueva, s/n

18071 Granada (Spain)

\texttt{magdarp@ugr.es}

\medskip

\textsc{Harold Rosenberg}

Instituto de Matematica Pura y Aplicada,

110 Estrada Dona Castorina,

22460-320 Rio de Janeiro, Brazil

\texttt{hrosen@free.fr}
\end{document}